\documentclass[11pt]{article}

\usepackage{comment}
\usepackage{fullpage}
\usepackage{amsmath,amsthm,amssymb}
\usepackage{mathrsfs,nicefrac}
\usepackage{amssymb}
\usepackage{epsfig}
\usepackage[all]{xy}
\usepackage{tocloft}
\usepackage{cancel}
\usepackage[strict]{changepage}
\ifx\dontloadhyperref\undefined
\usepackage[pdftex]{hyperref}
\usepackage{relsize}
\usepackage{color}
\usepackage{pdflscape}
\usepackage{graphicx}
\usepackage{stmaryrd}
\usepackage{cite} 

\usepackage{pgf,tikz}
\usetikzlibrary{arrows}

\else

\fi

\ifx\dontloaddefinitionsoftheoremenvironments\undefined
\theoremstyle{plain}
\newtheorem{thm}[equation]{Theorem}
\newtheorem*{thm*}{Theorem}
\newtheorem{lem}[equation]{Lemma}
\newtheorem*{lem*}{Lemma}
\newtheorem{prop}[equation]{Proposition}
\newtheorem*{prop*}{Proposition}
\newtheorem{cor}[equation]{Corollary}
\newtheorem*{cor*}{Corollary}

\theoremstyle{definition}

\newtheorem*{defn*}{Definition}
\else\relax\fi

\DeclareSymbolFont{AMSb}{U}{msb}{m}{n}
\DeclareMathSymbol{\N}{\mathbin}{AMSb}{"4E}
\DeclareMathSymbol{\Z}{\mathbin}{AMSb}{"5A}
\DeclareMathSymbol{\R}{\mathbin}{AMSb}{"52}
\DeclareMathSymbol{\Q}{\mathbin}{AMSb}{"51}
\DeclareMathSymbol{\I}{\mathbin}{AMSb}{"49}
\DeclareMathSymbol{\C}{\mathbin}{AMSb}{"43}
\DeclareMathSymbol{\A}{\mathbin}{AMSb}{"41}
\DeclareMathSymbol{\F}{\mathbin}{AMSb}{"46}

\renewcommand{\ker}{\textup{ker}\,}

\DeclareMathOperator{\im}{im}

\newcommand{\Squ}{\textup{Sq}}

\renewcommand{\mapsto}{\longmapsto}

\let\phi\varphi
\renewcommand{\to}{\longrightarrow}

\newcommand{\Addresses}{{
  \bigskip
  \footnotesize

 \textsc{Department of Mathematics, University of California, Los Angeles, CA
90095}\par\nopagebreak
  \textit{E-mail address}: \texttt{mjandr@math.ucla.edu}

  \medskip

 \textsc{Department of Mathematics, Massachusetts Institute of Technology, 
Cambridge, MA 02139}\par\nopagebreak
  \textit{E-mail address}: \texttt{hrm@math.mit.edu}
}}

\def\be{\begin{equation}}
\def\ee{\end{equation}}
\def\oa{\overline{\alpha}}
\def\V{\mathbb{M}_2}

\def\oBP{\overline{BP}{}}
\def\oH{\overline{H}{}}
\def\Mot{\text{Mot}}
\def\SMot{{\mathbb{S}_{\mathrm{Mot}}}}
\def\STop{\mathbb{S}}
\def\ZZ{\mathbb{Z}}
\def\CC{\mathbb{C}}

\begin{document}

	\title{Inverting the Hopf map}

	\author{Michael Andrews and Haynes Miller}

	\maketitle

\begin{abstract}
We calculate the $\eta$-localization of the motivic stable homotopy ring over $\C$, confirming a conjecture of Guillou and Isaksen. Our approach is via the motivic Adams-Novikov spectral sequence. In fact, work of Hu, Kriz, and Ormsby implies that it suffices to compute the corresponding localization of the classical Adams-Novikov $E_2$-term, and this is what we do. Guillou and Isaksen also propose a pattern of differentials in the localized motivic classical Adams spectral sequence, which we verify using a method first explored by Novikov.
\end{abstract}

\begin{center}
Dedicated to the memory of Goro Nishida (1943-2014)
\end{center}

\numberwithin{equation}{subsection}
\section{Introduction}
\subsection{Overview}

The chromatic approach to stable homotopy theory \cite{hopkins1998nilpotence} rests on the fact that non-nilpotent graded endomorphisms of finite complexes can be essentially classified. 
They are always detected by $MU$ (or, working locally at a prime, by $BP$), and up to taking $p^{\text{th}}$ powers every graded $BP_*BP$-comodule endomorphism survives the Adams-Novikov spectral sequence. 
At the prime $2$, for example, the Hopf map $\eta\in\pi_1(\STop)$ lies in filtration one, and the celebrated nilpotence theorem of Nishida \cite{nishida1973nilpotency} already guarantees that $\eta$ is nilpotent; in fact, we know that $\eta^4=0$. 

On the other hand, the element $\oa_1$ that detects $\eta$ in the Adams-Novikov $E_2$-term $E_2(\STop;BP)$ is non-nilpotent.
This immediate failure of the Adams-Novikov $E_2$-term to accurately reflect $2$-primary stable homotopy even in low degrees has lessened its attractiveness as a computational tool at the prime 2.
One can nevertheless hope to calculate $\oa_1^{-1}E_2(\STop;BP)$ and discover the range in which the localization map is an isomorphism, and this is our first main result.

Some $\oa_1$-free elements have been known for many years. 
The group $E_2^{1,2n}(\STop;BP)$ is cyclic for $n\geq 1$, generated by elements $\oa_n$ closely related to the image of $J$ \cite[Theorem $11.2$]{novikov1967methods}, \cite[Corollary $4.23$]{miller1977periodic}; 
in particular, $\oa_4$ detects $\sigma\in\pi_7(\STop)$.
In \cite{miller1977periodic} it was shown that for $n\neq 2$, $\oa_1^k\oa_n\neq 0$ for all $k\geq 0$. 
The Adams-Novikov differentials on these classes are also well-known and due to Novikov \cite{novikov1967methods}, \cite[p. $171$]{ravenel2004complex}. 
In this paper we show that there are no other $\oa_1$-free generators. A more precise statement is that
\begin{align}
\label{locANE_2}
\oa_1^{-1}E_2(\STop;BP)=
\F_2[\oa_1^{\pm 1},\oa_3,\oa_4]/
(\oa_4^2).
\end{align}
The complete structure of the localized Novikov spectral sequence now
follows from the differential 
\be
d_3\oa_3=\oa_1^4\,.
\label{novd3}
\ee
Since $\oa_1$ is a unit, this differential terminates the localized spectral sequence.

We also find that the localization map
$E_2(\STop;BP)\to\oa_1^{-1}E_2(\STop;BP)$
is an isomorphism above a line of slope $1/5$ when we plot the Adams-Novikov spectral sequence in the usual manner. 
This resolves a question raised by Zahler \cite{zahler1972adams} at the dawn of the chromatic era.

This result is of secondary interest from the perspective of the classical 
homotopy groups of spheres because we already know that $\eta^4=0$. 
However, the advent of motivic homotopy theory has led to interesting related questions. 
There is a ground field in motivic homotopy theory, which for us will be 
the complex numbers, and homotopy is bigraded. 
It is known that the motivic $\eta\in\pi_{1,1}(\SMot)$ is non-nilpotent, and one may ask (as Dugger and Isaksen did \cite{dugger2010motivic}) to calculate $\eta^{-1}\pi_{*,*}(\SMot)$. The determination of this localization is our 
second major result. We will in fact compute 
$\pi_{*,*}(\eta^{-1}((\SMot)^{\wedge}_2))$, but
a simple argument (Lemma \ref{eta-local-completion}) shows that 
$\eta^{-1}\SMot\to\eta^{-1}((\SMot)^{\wedge}_2)$ is an equivalence so we get
an uncompleted result as well. 

We can describe the result in terms of elements in 
$\pi_{*,*}(\SMot)$. There are motivic ``Hopf maps''
\[
\eta\in\pi_{1,1}(\SMot)\,,\quad
\sigma\in\pi_{7,4}(\SMot)
\]
as well as a unique nonzero class $\mu_9\in\pi_{9,5}(\SMot)$
($\!\!$\cite{hornbostel}, 2.10). Then
\[
\label{etalocmotsphere}
\eta^{-1}\pi_{*,*}(\SMot)=\F_2[\eta^{\pm 1},\sigma,\mu_9]/(\sigma^2)
\]

This extends work of Morel \cite{morel}, who verified this calculation 
(and extended it to a general ground field) for coweight zero, and of  
Hu, Kriz, and Ormsby \cite{hu2011remarks}, who showed that the localization
is at least this big, and it verifies a conjecture of Guillou and Isaksen 
\cite{guillou2014eta}. In subsequent work \cite{guillou2016etareal} these
authors carry out the analogous computation over the reals, starting with
the result over $\C$ obtained here. 

\subsection{Methods}

In our approach to \eqref{locANE_2} we
follow Novikov \cite{novikov1967methods} as interpreted in \cite{miller1981relations}: we filter the $BP$ cobar construction by powers of the kernel of the augmentation $BP_*\rightarrow\F_2$. 
The resulting ``algebraic Novkiov spectral sequence'' has the form 
\[H^*(P;Q)\implies E_2(\STop;BP)\]
where $P$ is the Hopf subalgebra of squares in the dual Steenrod algebra $A$ and 
\[
Q=\text{gr}^*BP_*=\F_2[q_0,q_1,\ldots]
\]
is the associated graded of $BP_*$; $q_n$ is the class of the Hazewinkel generator $v_n$. 
In this spectral sequence the element $\oa_1$ is represented by the class of $[\xi_1^2]$, which, following the notational conventions in force at an odd prime, we denote by $h_0$. 
By inverting $h_0$ we obtain a 
``localized algebraic Novikov spectral sequence'' converging to 
$\oa_1^{-1}E_2(\STop;BP)$. 
Our main result is a complete description of this spectral sequence.

The $E_1$-page of the localized algebraic Novikov spectral sequence is given by $h_0^{-1}H^*(P;Q)$. 
The computation of this object parallels the well-known fact \cite{may1981bockstein, miller1978localization} that if $M$ is a bounded below comodule over $A$ then the localized cohomology of $A$ with coefficients in $M$ can be computed
in terms of the homology of $M$ with respect to the differential $\Squ^1$:
\[q_0^{-1}H^*(A;M)=H(M;\text{Sq}^1)\otimes\F_2[q_0^{\pm 1}]\]
where $q_0\in H^{1,1}(A)$ is the class dual to $\text{Sq}^1$. 
The result we obtain is
\[
h_0^{-1}H^*(P;Q)=\F_2[h_0^{\pm 1},q_1^2,q_2,q_3,\ldots]\,.
\]

As in \cite{miller1981relations}, the differentials in the localized algebraic Novikov spectral sequence are calculated by exploiting the connection between $E_2(\STop;BP)$ and the theory of formal groups. 
The facts we rely on are easier to obtain than those used in \cite{miller1981relations} and date back to \cite{miller1976novikov}. 
We prove that 
\begin{align}\label{localgNovdiff}
d_1q_{n+1}=q_n^2h_0\quad \hbox{for }\, n\geq 2\,.
\end{align}
This is the last possible differential, and we arrive at \eqref{locANE_2}.
It follows, incidentally, that for $m,n\geq 0$, $\oa_{2m}\oa_{2n}$ is $\oa_1$-torsion, and that modulo $\oa_1$-torsion the classes $\oa_{2m+1}\oa_{2n+1}$ and $\oa_{2m+1}\oa_{2n}$ depend only on the sum $m+n$. 

This computation leads without difficulty to the motivic result.
Hu, Kriz and Ormsby ($\!\!$\cite{hu2011remarks,hu2011convergence}; 
see also \cite{dugger2010motivic}) study a motivic analogue of the Adams-Novikov spectral sequence. To circumvent incomplete understanding of the motivic analogue of $BP$, they work with its 2-adic completion, which we will denote by
$BPM$. They show that its $E_2$-term is 
\[
E_2(\SMot;BPM)=E_2(\STop;BP)\otimes\ZZ_2[\tau]\,,
\] 
where $\tau$ detects a certain class in 
$\pi_{0,-1}((\SMot)^{\wedge}_2)$. 
Inverting $\oa_1$ results in a spectral sequence which we show converges to
$\pi_{*,*}(\eta^{-1}\SMot)=\pi_{*,*}((\SMot)^\wedge_2)$. 
The unique nonzero differential is
determined by the motivic analogue of \eqref{novd3}, namely
\begin{align}\label{locANdiff}
d_3\oa_3=\tau\oa_1^4\,,
\end{align}
and \eqref{etalocmotsphere} follows.

Our calculation of $\eta^{-1}\pi_{*,*}(\SMot)$ verifies the conjecture
of Guillou and Isaksen \cite{guillou2014eta}). Those authors
approached this calculation by means of the motivic classical Adams spectral sequence. 
Using the motivic May spectral sequence they found that
\begin{align}\label{gi-comp}
h_0^{-1}E_2(\SMot;H)=
\F_2[h_0^{\pm 1},v_1^4,v_2,v_3,\ldots]\,.
\end{align}
Based on extensive computations they conjectured that 
in the localized motivic Adams spectral sequence 
\begin{align}\label{gi-diff}
d_2v_{n+1}=v_n^2h_0\,, n\geq 2\,.
\end{align}
In the last part of this paper we recover 
\eqref{gi-comp} using a localized Cartan-Eilenberg spectral sequence. 
We then use the methods of \cite{miller1981relations} to show that
the Adams differentials \eqref{gi-diff} follow from the differentials 
\eqref{localgNovdiff} in the algebraic Novikov spectral sequence.

\subsection{Organization}
In Section \ref{secane2} we recall the form of the Adams-Novikov 
$E_2$ term and the definition of the elements 
$\oa_i\in E_2^{1,2i+1}(\STop;BP)$.
Next we recall the algebraic Novikov spectral sequence, and in 
Section \ref{seclocane1} we compute the localization of its initial term
and verify vanishing lines required to check its convergence. 
These vanishing lines give us a range of dimensions in which the
localization map is an isomorphism. In Section \ref{seclocanssdiff} 
we construct some permanent cycles, and then bring the theory 
of formal groups into play to compute the differential in the localized
algebraic Novikov spectral sequence and verify \eqref{locANE_2}.

In Section \ref{secmot} we recall the basic setup for motivic stable 
homotopy theory over $\CC$, and prove \eqref{etalocmotsphere}. 
In the final section,
we provide an improved statement and proof of the 
comparison of differentials in the Adams and algebraic Novikov spectral 
sequences ($\!\!$\cite{novikov1967methods, miller1981relations}), 
and observe that the 
local algebraic Novikov differentials then imply the local
motivic Adams differentials conjectured by Guillou and Isaksen.

\subsection{Acknowledgments}
The first author is particularly grateful to Dan Isaksen and Zhouli Xu for their talks on the motivic Adams spectral sequence and its relationship with the Adams-Novikov spectral sequence. 
These served as the inspiration to start trying to compute the Adams-Novikov $E_2$-page. 
We thank Amelia Perry, who created charts of the algebraic Novikov $E_1$-page. 
The idea for the proof of the main theorem came about while she coded computational software for us. 
We are indebted to Dan Isaksen for raising these questions about the $\eta$-local motivic sphere and for several useful conversations. 
We thank him and Bert Guillou for keeping us abreast of their work, and
Dan Dugger, Dan Isaksen, Marc Levine, and Kyle Ormsby
for helpful tutorials on motivic matters.
We are also indebted to Aravind Asok for pointing out oversight in an earlier
draft, to Jens Hornbostel for providing a fix, and to the referee for pointing
out that the completion map induces an isomorphism after inverting $\eta$;
and to Lyubo Panchev for his careful reading of Section 9.

\numberwithin{equation}{section}
\section{The Adams-Novikov spectral sequence}
\label{secane2}

In this paper we follow \cite{miller1977periodic} and work with \emph{right} comodules. 
Given a Hopf algebroid $(A,\Gamma)$ and a right $\Gamma$-comodule $M$, the \emph{cobar construction} $\Omega^*(\Gamma;M)$ has $\Omega^s(\Gamma;M)=
M\otimes_A\overline{\Gamma}\otimes_A\cdots\otimes_A\overline{\Gamma}$ with $s$ copies of $\overline{\Gamma}=\ker(\epsilon:\Gamma\rightarrow A)$, and is equipped with a natural differential \cite[($1.10$)]{miller1977periodic} of degree $1$. 
$H^*(\Gamma;M)$ denotes the cohomology of this complex. 
If $\Gamma$ and $M$ are graded then $H^*(\Gamma;M)$ becomes bigraded; the first index is the cohomological grading and the second is inherited from the gradings on $\Gamma$ and $M$.

Recall (e.g. \cite[\S 2]{miller1976novikov}) the Hopf algebroid given to us by the $p$-typical factor of complex cobordism: $(BP_*,BP_*BP)$. We will work at the prime $p=2$ throughout this paper. The Adams-Novikov 
spectral sequence for a connective spectrum $X$ takes the following form:
\[
E_2^{s,u}=H^{s,u}(BP_*BP;BP_*(X))
{\implies}
\pi_{u-s}(X)\otimes\Z_{(2)},\
d_r:E_r^{s,u}\to E_r^{s+r,u+r-1}\,.
\]

The coefficient ring of $BP$ is a polynomial algebra, 
\[
BP_*=\Z_{(2)}[v_1,v_2,\ldots]\,,\quad|v_i|=2(2^i-1)\,.
\]
As in \cite{miller1977periodic}, we have a short exact sequence of $BP_*BP$-comodules
\[
0\to BP_*\to 2^{-1}BP_*\to BP_*/2^{\infty}\to 0\,,
\]
which gives rise to a connecting homomorphism
\[
\delta:H^{0,*}(BP_*BP;BP_*/2^{\infty})\to H^{1,*}(BP_*BP;BP_*)\,.
\]
Following \cite{miller1977periodic}, 
define $x=v_1^2-4v_1^{-1}v_2\in 2^{-1}v_1^{-1}BP_*$. 
When $s\geq 2$, the image of $x^s/8s$ in $v_1^{-1}BP_*/2^{\infty}$ lies in the subgroup $BP_*/2^{\infty}$. We define
\begin{align*}
\oa_s=
\begin{cases}
\delta (v_1^s/2)   & s\geq 1\text{ and odd}\\
\delta (v_1^2/4)   & s=2\\
\delta (x^{s/2}/4s)& s\neq 2\text{ and even}\,.\\
\end{cases}
\end{align*}

\begin{prop}[{\cite{miller1977periodic}}]
These classes generate $H^{1,*}(BP_*BP)$ and are of the same $2$-order as the denominator of the element to which $\delta$ was applied.
\end{prop}

The element $\oa_1$ is a permanent cycle detecting $\eta\in\pi_1(\STop)$. In Lemma \ref{alpha_1-free} we will show that for 
$s\neq 2$, $\oa_s$ is not killed by any power of $\oa_1$.
This is \cite[Corollary 4.23]{miller1977periodic}, but our proof is different. 
We will analyze the images of these elements in 
$\oa_1^{-1}H^*(BP_*BP)$, an object we will compute explicitly. 
We show that these classes, and 1, generate 
$\oa_1^{-1}H^*(BP_*BP)$ as an 
$\F_2[\oa_1^{\pm 1}]$-vector space.

\section{The algebraic Novikov spectral sequence}
\label{secanss}
One of the best tools for gaining information about the $E_2$-page of the Adams-Novikov spectral sequence is the algebraic Novikov spectral sequence
\cite{novikov1967methods,miller1981relations}, which arises from filtering the
cobar complex $\Omega^*(BP_*BP)$ by powers of the augmentation ideal
\[
I=\ker(BP_*\rightarrow\F_2)\,.
\]

In this paper we will write $A$ for the dual of the Steenrod algebra and 
$P$ for the Hopf subalgebra of squares in $A$. Write
\[\zeta_n=\overline{\xi}_n^2\]
for the square of the conjugate of the Milnor generator $\xi_n$, so that
\[P=\F_2[\zeta_1,\zeta_2,\ldots]\]
with diagonal
\[\Delta\zeta_n=\sum_{i+j=n}\zeta_i\otimes\zeta_j^{2^i}.\]
Define the graded algebra in $P$-comodules
\[
Q=\F_2[q_0,q_1,\ldots] \,,\quad|q_i|=(1,2(2^i-1))\,,
\]
with coaction defined by 
\[
q_n\mapsto\sum_{i+j=n}q_i\otimes\zeta_j^{2^i}\,.
\]
Write $Q^t$ for the component of $Q$ with first gradation $t$; this is the
``Novikov degree.'' Then \cite{miller1981relations} 
\begin{align}\label{assgraded}
\text{gr}^t\Omega^s(BP_*BP)=\Omega^s(P;Q^t)
\end{align}
and the algebraic Novikov spectral sequence takes the form
\begin{align}\label{alg.nss}
E_1^{s,t,u}=H^{s,u}(P;Q^t)
{\implies} H^{s,u}(BP_*BP)\,,\quad
d_r:E_r^{s,t,u}\to E_r^{s+1,t+r,u}\,.
\end{align}

Certain elements will be important for us. We will abbreviate the element in $H^{0,0}(P;Q^1)$ represented by $q_0[\ ]$ to $q_0$, and write $h_n$ for the element in $H^{1,2^{n+1}}(P;Q^0)$ represented by $[\zeta_1^{2^n}]$.

The algebra 
$H^*(P;Q)$ is not only the $E_1$-page for the algebraic Novikov spectral sequence, but also the $E_2$-page for the Cartan-Eilenberg spectral sequence
associated to the extension of Hopf algebras
\[P\to A\to E,\]
where $E$ denotes the exterior algebra 
$E[\overline{\xi}_1,\overline{\xi}_2,\ldots]$. 
The $E_2$ term has the form $H^*(P;H^*(E))$, 
and 
\[
H^*(E)=\F_2[q_0,q_1,\ldots]=Q\,,\quad q_{n-1}=[\overline{\xi}_n]\,.
\]
The coaction of $P$ on $Q$ coincides with the action described above.
The extension spectral sequence 
converges to the $E_2$-page of the Adams spectral sequence:
\begin{align}\label{CESS}
E_2^{s,t,u}=H^{s,u}(P;Q^t)
{\implies} H^{s+t,u+t}(A),\ 
d_r:E_r^{s,t,u}\to E_r^{s+r,t-r+1,u+r-1}\,.
\end{align}
This leads to the following square of spectral sequences
\cite{novikov1967methods,miller1981relations}:

\be
\xymatrixrowsep{38pt}\xymatrix{
H^{s,u}(P;Q^t) \ar@{=>}[rr]^{\text{CESS}}
\ar@{=>}[d]_{\text{ANSS}} 
&& E_2^{s+t,u+t}(\STop;H)\ar@{=>}[d]^{\text{ASS}} \\
E_2^{s,u}(\STop;;BP)\ar@{=>}[rr]^{\text{NSS}}&& \pi_{u-s}(\STop)
}
\label{ss-square}
\ee

It is useful to keep in mind two projections of the trigraded object 
$H^*(P;Q^*)$,
corresponding to the two spectral sequences of which it is an initial term. 
The ``Novikov projection'' displays $(u-s,s)$ and suppresses the filtration 
grading in the algebraic Novikov spectral sequence (Novikov weight); 
it presents the spectral sequence as it will appear in $E_2(\STop;BP)$. 
The ``Adams projection'' displays $(u-s,s+t)$, and suppresses the filtration
grading $s$ in the Cartan-Eilenberg spectral sequence; it presents the 
spectral sequence as it will appear in $E_2(\STop;H)$.

In Figure \ref{fig:H(P;Q)grading} we have shown the low-dimensional part of 
these two projections.
Vertical black lines indicate multiplication by $q_0$. 
The vertical blue arrow indicates a $q_0$ tower which continues indefinitely. 
Black lines of slope one indicate multiplication by $h_0$. 
The blue arrows of slope one indicate $h_0$ towers which continue indefinitely. Green arrows denote algebraic Novikov differentials and red arrows denote Cartan-Eilenberg differentials. 
On the first chart square nodes denote multiple basis elements connected by $q_0$-multiplication; the number to the upper left indicates how many such basis elements.

\begin{figure}
\centering

\begin{tikzpicture}[x=0.6cm,y=0.6cm]
\draw[->,color=black] (0,0) -- (15.5,0);
\foreach \x in {0,2,4,6,8,10,12,14}
\draw[shift={(\x,0)},color=black] (0pt,-2pt) -- (0pt,0pt) node[below] {\footnotesize $\x$};
\draw[->,color=black] (0,0) -- (0,9);
\foreach \y in {0,2,4,6,8}
\draw[shift={(0,\y)},color=black] (-2pt,0pt) -- (0pt,0pt) node[left] {\footnotesize $\y$};
\node at (15.5,-1) {$u-s$};
\node at (-1.5,9) {$s+t$};

\draw [->,red] (5,3) -- (4,4);
\draw [->,red] (11,4) -- (10,5);
\draw [->,red] (13,7) -- (12,8);
\draw [->,green] (15,1) -- (14,3);
\draw [->,green] (15,2) -- (14,5);
\draw [->,green] (15,3) -- (14,6);

\node[black,fill,circle,inner sep=0pt,minimum size=2pt] at (0,0) {};

\node[black,fill,circle,inner sep=0pt,minimum size=2pt] at (0,1) {};
\draw [->,blue] (0,2) -- (0,2.8);
\node[black,fill,circle,inner sep=0pt,minimum size=2pt] at (0,2) {};

\draw [black] (0,0) -- (1,1);
\draw [->,blue] (1,1) -- (1.8,1.8);
\node[black,fill,circle,inner sep=0pt,minimum size=2pt] at (1,1) {};
\draw [gray,dashed] (1.8,1.8) -- (4.8,4.8);

\draw [black] (3,1) -- (3,2);
\node[black,fill,circle,inner sep=0pt,minimum size=2pt] at (3,1) {};
\node[black,fill,circle,inner sep=0pt,minimum size=2pt] at (3,2) {};

\draw [->,blue] (5,3) -- (5.8,3.8);
\node[black,fill,circle,inner sep=0pt,minimum size=2pt] at (5,3) {};

\node[black,fill,circle,inner sep=0pt,minimum size=2pt] at (6,2) {};

\draw [black] (7,1) -- (7,4);
\draw [black] (7,1) -- (9,3);

\foreach \x in {1,2,3,4}
\node[black,fill,circle,inner sep=0pt,minimum size=2pt] at (7,\x) {};

\node[black,fill,circle,inner sep=0pt,minimum size=2pt] at (8,2) {};
\node[black,fill,circle,inner sep=0pt,minimum size=2pt] at (9,3) {};

\draw [->,blue] (8,3) -- (8.8,3.8);
\node[black,fill,circle,inner sep=0pt,minimum size=2pt] at (8,3) {};
\draw [gray,dashed] (8.8,3.8) -- (10.8,5.8);

\draw [->,blue] (9,5) -- (9.8,5.8);
\node[black,fill,circle,inner sep=0pt,minimum size=2pt] at (9,5) {};
\draw [gray,dashed] (9.8,5.8) -- (12.8,8.8);

\draw [black] (11,4) -- (11,6);
\draw [->,blue] (11,4) -- (11.8,4.8);
\foreach \x in {4,5,6}
\node[black,fill,circle,inner sep=0pt,minimum size=2pt] at (11,\x) {};

\draw [->,blue] (13,7) -- (13.8,7.8);
\node[black,fill,circle,inner sep=0pt,minimum size=2pt] at (13,7) {};

\draw [black] (14,2) -- (14,3);
\node[black,fill,circle,inner sep=0pt,minimum size=2pt] at (14,2) {};
\node[black,fill,circle,inner sep=0pt,minimum size=2pt] at (14,3) {};

\draw [black] (14,4) -- (14,6);
\draw [->,blue] (14,4) -- (14.8,4.8);
\foreach \x in {4,5,6}
\node[black,fill,circle,inner sep=0pt,minimum size=2pt] at (14,\x) {};

\draw [black] (15,1) -- (15,8);
\draw [black] (15,1) -- (15.5,1.5);
\foreach \x in {1,2,3,4,5,6,7,8}
\node[black,fill,circle,inner sep=0pt,minimum size=2pt] at (15,\x) {};

\node at (0.28, 0.88) {\scalebox{.6}{$q_0$}};
\node at (1.28, 0.88) {\scalebox{.6}{$h_0$}};
\node at (3.28, 0.88) {\scalebox{.6}{$h_1$}};
\node at (6.28, 1.88) {\scalebox{.6}{$h_1^2$}};
\node at (7.28, 0.88) {\scalebox{.6}{$h_2$}};
\node at (14.28, 1.88) {\scalebox{.6}{$h_2^2$}};
\node at (15.28, 0.88) {\scalebox{.6}{$h_3$}};

\node at (4.2, 5.88) {\scalebox{.6}{$P=\langle h_1,q_0^2,-\rangle$}};

\node at (5.4, 2.88) {\scalebox{.6}{$Ph_0$}};
\node at (9, 4.72) {\scalebox{.6}{$P^2h_0$}};
\node at (13.48, 6.88) {\scalebox{.6}{$P^3h_0$}};

\node at (8, 2.8) {\scalebox{.4}{$\langle h_0,q_0,h_1^2\rangle$}};

\begin{scope}[shift={(0,12)}]
\draw[->,color=black] (0,0) -- (15.5,0);
\foreach \x in {0,2,4,6,8,10,12,14}
\draw[shift={(\x,0)},color=black] (0pt,-2pt) -- (0pt,0pt) node[below] {\footnotesize $\x$};
\draw[->,color=black] (0,0) -- (0,5);
\foreach \y in {0,2,4}
\draw[shift={(0,\y)},color=black] (-2pt,0pt) -- (0pt,0pt) node[left] {\footnotesize $\y$};
\node at (15.5,-1) {$u-s$};
\node at (-1,5) {$s$};

\draw [->,red] (5,1) -- (4,4);
\draw [->,red] (11,1) -- (10,4.2);
\draw [->,red] (13,1) -- (12,4);
\draw [->,green] (15,1) -- (14,1.8);
\draw [->,green] (15,1) -- (14,2.2);

\node[black,fill,rectangle,inner sep=0pt,minimum size=2pt] at (0,0) {};
\node at (-0.12, 0.12) {\scalebox{.4}{$\infty$}};

\draw [black] (0,0) -- (1,1);
\draw [->,blue] (1,1) -- (1.8,1.8);
\node[black,fill,circle,inner sep=0pt,minimum size=2pt] at (1,1) {};
\draw [gray,dashed] (1.8,1.8) -- (4.8,4.8);

\node[black,fill,rectangle,inner sep=0pt,minimum size=2pt] at (3,1) {};
\node at (2.88, 1.12) {\scalebox{.4}{$2$}};

\draw [->,blue] (5,1) -- (5.8,1.8);
\node[black,fill,circle,inner sep=0pt,minimum size=2pt] at (5,1) {};

\node[black,fill,circle,inner sep=0pt,minimum size=2pt] at (6,2) {};

\draw [black] (7,1) -- (9,3);

\node[black,fill,rectangle,inner sep=0pt,minimum size=2pt] at (7,1) {};
\node at (6.88, 1.12) {\scalebox{.4}{$4$}};

\node[black,fill,circle,inner sep=0pt,minimum size=2pt] at (8,2) {};
\node[black,fill,circle,inner sep=0pt,minimum size=2pt] at (9,3) {};

\draw [->,blue] (8,2.2) -- (8.8,3);
\node[black,fill,circle,inner sep=0pt,minimum size=2pt] at (8,2.2) {};
\draw [gray,dashed] (8.8,3) -- (10.8,5);

\draw [->,blue] (9,1) -- (9.8,1.8);
\node[black,fill,circle,inner sep=0pt,minimum size=2pt] at (9,1) {};
\draw [gray,dashed] (9.8,1.8) -- (12.8,4.8);

\draw [->,blue] (11,1) -- (11.8,1.8);
\node[black,fill,rectangle,inner sep=0pt,minimum size=2pt] at (11,1) {};
\node at (10.88, 1.12) {\scalebox{.4}{$3$}};

\draw [->,blue] (13,1) -- (13.8,1.8);
\node[black,fill,circle,inner sep=0pt,minimum size=2pt] at (13,1) {};

\node[black,fill,rectangle,inner sep=0pt,minimum size=2pt] at (14,1.8) {};
\node at (13.88, 1.92) {\scalebox{.4}{$2$}};

\draw [->,blue] (14,2.2) -- (14.8,3);
\node[black,fill,rectangle,inner sep=0pt,minimum size=2pt] at (14,2.2) {};
\node at (13.88, 2.32) {\scalebox{.4}{$3$}};

\draw [black] (15,1) -- (15.5,1.5);
\node[black,fill,rectangle,inner sep=0pt,minimum size=2pt] at (15,1) {};
\node at (14.88, 1.12) {\scalebox{.4}{$8$}};

\node at (1.28, 0.88) {\scalebox{.6}{$h_0$}};
\node at (3.28, 0.88) {\scalebox{.6}{$h_1$}};
\node at (6.28, 1.88) {\scalebox{.6}{$h_1^2$}};
\node at (7.28, 0.88) {\scalebox{.6}{$h_2$}};
\node at (14.28, 1.68) {\scalebox{.6}{$h_2^2$}};
\node at (15.28, 0.88) {\scalebox{.6}{$h_3$}};

\node at (5.4, 0.88) {\scalebox{.6}{$Ph_0$}};
\node at (9.48, 0.88) {\scalebox{.6}{$P^2h_0$}};
\node at (13.48, 0.88) {\scalebox{.6}{$P^3h_0$}};

\node at (7.38, 2.38) {\scalebox{.4}{$\langle h_0,q_0,h_1^2\rangle$}};
\end{scope}

\end{tikzpicture}

\caption{$H^*(P;Q)$ in Novikov and classical Adams projections. 
Green arrows are algebraic Novikov differentials and red arrows are Cartan-Eilenberg differentials.}
\label{fig:H(P;Q)grading}
\end{figure}
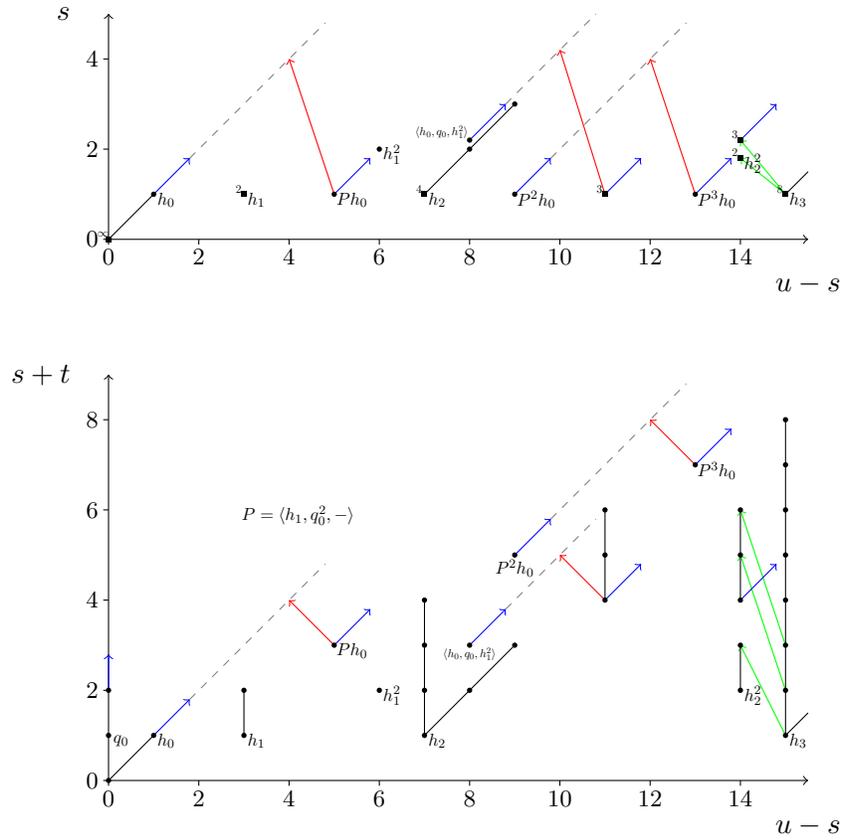

\section{Localizing the algebraic Novikov $E_1$-page.}
\label{seclocane1}

We wish to localize the algebraic Novikov spectral sequence by inverting $h_0$. 
We begin with a well-known 
localization theorem, dealing with comodules over the dual Steenrod algebra $A$. In working with ordinary homology we will work with left comodules, to be
consistent with our presentation of the proof of Theorem \ref{thmcomparison}
below, though of course the categories of left comodules and right comodules
are equivalent via the anti-automorphism in the Hopf algebra $A$.
Write $q_0$ for the class of $[\xi_1]$ in $H^{1,1}(A)$. It acts on $H^*(A;M)$ for any $A$-comodule $M$. Write $E$ for the quotient Hopf algebra 
\[
E=A/(\xi_1^2,\xi_2,\xi_3,\ldots)\,.
\]
It is the exterior algebra generated by the image of $\xi_1$. Any $A$-comodule $M$ becomes an $E$-comodule, and an $E$-comodule structure on $M$ is equivalent to a degree $-1$ differential $\text{Sq}^1$ on $M$ given in terms of the
coaction by
\[
x\mapsto 1\otimes x+\xi_1\otimes x\text{Sq}^1\,.
\]

\begin{prop}
Let $M$ be an $A$-comodule such that $M_u=0$ whenever $u<0$, and consider the following diagram.
\[\xymatrix{H^*(A;M)\ar[r]\ar[d]& H^*(E;M)\ar[d]\\ 
q_0^{-1}H^*(A;M)\ar[r]& q_0^{-1}H^*(E;M)}\]
The top map is surjective in bidegrees $(s,u)$ with $u-s<2s-2$ and an isomorphism in bidegrees with $u-s<2s-5$, i.e. above a line of slope $1/2$ in the usual $(u-s,s)$ plot. The bottom map is an isomorphism and the right map is an isomorphism for bidegrees $(s,u)$ with $s>0$. Moreover,
\[
q_0^{-1}H^*(E;M)=H(M;\text{Sq}^1)\otimes\F_2[q_0^{\pm 1}]\,.
\]
\end{prop}

\begin{proof}
The cotensor product $A\Box_E M$ is a submodule of $A\otimes M$. Since the coaction map $M\to A\otimes M$ is associative, it factors through a map 
$i:M\to A\Box_E M$. Define $L$ by the following short exact sequence of
$A$-comodules.
\begin{align}\label{L}
\xymatrix{0\ar[r]& M\ar[r]^-i& A\Box_E M\ar[r]& L\ar[r]& 0}
\end{align}
We claim that $H^{s,u}(A;L)=0$ whenever $u-s<2s-2$.

If $M=\F_2$, the middle comodule $A\Box_E\F_2$ is the homology of the integral Eilenberg Mac Lane spectrum. It is well known, in that case, that the map $i$ induces an isomorphism in $\text{Sq}^1$-homology. (One way to see this is to think about the dual: left multiplication by $\text{Sq}^1$ gives a bijection between the Cartan-Serre basis elements for $H^*(H\Z)$ with even leading entry and those with odd leading entry, with the exception of the basis element $1$ in dimension $0$.) Filtering the general comodule $M$ by dimension shows that the same is true in general. We deduce that $L$ is $\text{Sq}^1$-acyclic and so we can apply \cite[Theorem $2.1$]{adams1966periodicity} or \cite[Theorem $1.1$]{anderson1973vanishing} to give the claimed vanishing line for $H^*(A;L)$.

Under the identification $H^*(A;A\Box_E M)=H^*(E;M)$, the map induced by applying $H(A;-)$ to $i$ is the top map in the proposition statement and so the first statement in the proposition follows from the cohomology long exact sequence associated to \eqref{L} and the vanishing line just proved. 

Since $q_0$ acts vertically in $(u-s,s)$ coordinates we find that the bottom map is an isomorphism. The remaining statements follow from the identification
\[
H^*(E;M)=
\frac{\ker(\text{Sq}^1)\otimes\F_2[q_0]}
{\im(\text{Sq}^1)\otimes(q_0)}.
\]
\end{proof}

Thus the localization of the Adams $E_2$-page coincides with the $E_2$-page of the Bockstein spectral sequence. 
In fact \cite{may1981bockstein,miller1981relations} the two spectral sequences coincide from $E_2$ onwards, giving a qualitative strengthening of Serre's observation that $\pi_*(X)\otimes\Q\cong H_*(X;\Q)$.

By doubling degrees we obtain a parallel result for the Hopf subalgebra $P$ of $A$. Now $E$ will be the quotient Hopf algebra $P/(\zeta_1^2,\zeta_2,\ldots)$. Any $P$-comodule $M$ becomes an $E$-comodule, and just as we wrote $\text{Sq}^1$ above we will write $P^1$ for the operator on a right $E$-comodule corresponding to $\zeta_1$. A $P$-comodule splits naturally into even and odd parts, which one can handle separately to prove the following result.

\begin{prop}\label{Ploc}
Let $M$ be a $P$-comodule such that $M_u=0$ whenever $u<0$, and consider the following diagram.
\[\xymatrix{H^*(P;M)\ar[r]\ar[d]& H^*(E;M)\ar[d]\\ 
h_0^{-1}H^*(P;M)\ar[r]& h_0^{-1}H^*(E;M)}\]
The top map is surjective in bidegrees $(s,u)$ with $u-s<5s-4$ and an isomorphism in bidegrees with $u-s<5s-10$, i.e. above a line of slope $1/5$ in the usual $(u-s,s)$ plot. The bottom map is an isomorphism and the right map is an isomorphism for bidegrees $(s,u)$ with $s>0$. Moreover,
\[
h_0^{-1}H^*(E;M)=\F_2[h_0^{\pm 1}]\otimes H(M;P^1)\,.
\]
\end{prop}

As an application, we obtain a calculation of the $h_0$-localization of the $E_1$-page of the algebraic Novikov spectral sequence together with the range in which the localization map is an isomorphism.

\begin{cor}\label{E_1loc}
For any $t$, the localization map $H^*(P;Q^t)\to h_0^{-1}H^*(P;Q^t)$ is surjective in bidegrees $(s,u)$ with $u-s<5s-4$ and an isomorphism in bidegrees with $u-s<5s-10$. Moreover, 
\[
h_0^{-1}H^*(P;Q)=\F_2[h_0^{\pm 1},q_1^2,q_2,q_3,\ldots]\,.
\]
\end{cor}

\begin{proof}
It is enough to note that $q_1P^1=q_0$, 
$\ker P^1=\F_2[q_0,q_1^2,q_2,q_3,\ldots]$, and $\im P^1=(q_0)$.
\end{proof}

In order to check convergence of the localized algebraic Novikov
spectral sequence we will require some basic vanishing lines, which 
are suggested by the diagrams in Figure \ref{fig:H(P;Q)grading}.
The first one is easy:

\begin{lem}\label{vanishingAdamsNov}
$H^{s,u}(P;Q^t)=0$ when $u-s<s$.
\end{lem}

\begin{proof}
The cobar construction has the form
\[
\Omega^{s,u}(P;M)=(\overline{P}^{\otimes s}\otimes M)_u
\] 
where $\overline{P}$ denotes the positive-dimensional part of $P$.
Since $\overline{P}_u=0$ for $u<2$, $\Omega^{s,u}(P;M)=0$ for $u<m+2s$
if $M_u=0$ for $u<m$. 

For any $t$, $Q^{t,u}=0$ for $u<0$. Thus
$\Omega^{s,u}(P;Q^t)=0$ for $u<2s$.
\end{proof}

This crude vanishing line can be improved when $u>s$:

\begin{lem}\label{vanishingAdams}
$H^{s,u}(P;Q^t)=0$ when $0<u-s<s+t$.
\end{lem}

\begin{proof}
The cobar complex itself vanishes if $u$, $s$, or $t$ is negative. The
groups $H^{0,*}(P;Q^t)$ constitute the primitives of $Q^t$, which are 
generated as 
a vector space by $q_0^t$. The constraint $0<u-s$ avoids these classes.

Define an algebra map $\phi:Q\to P$ by sending $q_n$ to $\zeta_n$ (so $q_0$ maps to $1$). Restricting to Novikov degree $t$ gives us an embedding of 
$P$-comodules, and, by the vanishing of $H^*(P;P)$, 
the boundary map $H^{s-1,u}(P;P/Q^t)\to H^{s,u}(P;Q^t)$ 
in the associated long exact sequence is surjective as long as $s>0$.
Now $(P/Q^t)_u=0$ if $u<2(t+1)$, since the first element not in the image of $\phi|_{Q^t}$ is $\zeta_1^{t+1}$. 
Thus $H^{s-1,u}(P;P/Q^t)$ is zero provided that $u<2(s-1)+2(t+1)=2s+2t$. 
Since $Q^t=0$ for $t<0$, we may assume $t\geq0$, and then $u-s<s+t$ implies 
$u<2s+2t$.
\end{proof}

\section{The localized algebraic Novikov spectral sequence}
\label{seclocanss}

The algebraic Novikov spectral sequence is multiplicative since it is obtained by filtering the $DG$ algebra $\Omega^*(BP_*BP)$ by powers of a differential ideal. Since the class $h_0\in E_1^{1,0,2}$ is a permanent cycle and 
$H^*(BP_*BP)$ is commutative, inverting $h_0$ gives 
a new multiplicative spectral sequence.

In forming this localization we may lose convergence; this is the issue at stake in the ``telescope conjecture'' of chromatic homotopy theory. Here we are lucky, however. Convergence is preserved because, as was the case in \cite{miller1981relations}, the operator we are inverting acts parallel to a vanishing line. This vanishing line is visible in the lower diagram in figure \ref{fig:H(P;Q)grading} and is the content of Proposition \ref{vanishingAdams}. 

The vanishing line has the following two implications for the algebraic Novikov
spectral sequence.
\begin{enumerate}
\item 
If $x\in H^*(BP_*BP)$ is not killed by any power of $\oa_1$, 
then for some $k$, $\oa_1^kx$ has a representative at $E_1$ that is not killed
by any power of $h_0$. 
\item 
For any $a\in E_1$, $h_0^ka$ is a permanent cycle for all sufficiently 
large $k$ (though it may be zero).
\end{enumerate}

The first fact tells us that we detect everything we are supposed to; the second fact tells us that we do not detect more than we are supposed to. A more thorough account of similar convergence issues may be found in \cite{andrews2013periodic}.

Coupled with the isomorphism range in Corollary \ref{E_1loc}, the natural map of spectral sequences from the algebraic Novikov spectral sequence to its localized counterpart implies an isomorphism range for the $\oa_1$-localization of the Adams-Novikov $E_2$-page for the sphere. 

\begin{prop}
The localization map
\[
H^*(BP_*BP)\to\oa_1^{-1}H^*(BP_*BP)
\]
is surjective in bidegrees $(s,u)$ for which $u-s<5s-4$ and an isomorphism in bidegrees for which $u-s<5s-10$, i.e. above a line of slope $1/5$ in the usual $(u-s,s)$ plot.
\end{prop}

\begin{proof}
We use the natural map of spectral sequences from the algebraic Novikov spectral sequence to its localized counterpart. 

First, Lemma \ref{vanishingAdams} implies that in each bidegree $(s,u)$ 
other than $(0,0)$, 
\[
E_1^{s,t,u}=0 \quad\hbox{for}\quad t>u-2s\,.
\]
The same then holds for $E_r$ for all $r\geq0$, 
so by convergence of the spectral sequence
\[
F^{u-2s}H^{s,u}(BP_*BP)=0 \quad\hbox{for}\quad(s,u)\neq(0,0)\,.
\]
Multiplication by $h_0$ preserves $u-2s$ and $t$, so the same facts hold
for the localized spectral sequence and for the filtration of 
$\oa_1^{-1}H^*(BP_*BP)$. 

Next, observe that Lemma \ref{vanishingAdams} implies that
\[
d_r(E_r^{s,t,u})=0 \quad\hbox{for}\quad r>u-2(s+1)-t
\]
and the fact that $Q^t=0$ for $t<0$ implies that 
\[
E_r^{s,t,u}\supseteq\im(d_r)=0\quad\hbox{for}\quad r>t\,,
\]
so in each tridegree the spectral sequence terminates at a finite stage. 

Finally, we claim that for each $r\geq1$ the map at $E_r$ is surjective in bidegrees $(s,u)$ for which $u-s<5s-4$ and an isomorphism in bidegrees for which $u-s<5s-10$. We know this to be true at the $E_1$-page by Proposition \ref{Ploc}; suppose it is true for the $E_r$-page. A $d_r$-differential in the algebraic Novikov spectral sequence has $(s,u)$ bidegree $(1,0)$. Because $u-s<5s-4$ if and only if $u-(s+1)<5(s+1)-10$, it has source in the surjective region if and only if it has target in the isomorphism region. Thus, we can deduce the result for the $E_{r+1}$-page using the following simple observation: if a map of cochain complexes is a surjection in degree $n$ and an isomorphism in higher degrees, then the same is true of the map induced in cohomology. 

The result now follows by an induction on the filtration. 
\end{proof}

\section{Computing the localized algebraic Novikov spectral sequence}
\label{seclocanssdiff}

We begin by identifying some permanent cycles in the algebraic Novikov
spectral sequence. Recall that 
\[
BP_*BP=BP_*[t_1,t_2,\ldots]\,,\quad|t_i|=2(2^i-1)\,.
\]

\begin{lem}\label{BPcocycles}
The following elements are cocycles in the cobar construction $\Omega^*(BP_*BP)$:
\begin{enumerate}
\item $[\ ]$ and $[t_1]$;
\item $v_1^2[t_1]+2v_1[t_1^2]+\frac{4}{3}[t_1^3]$;
\item $v_2[t_1|t_1]+v_1[t_1|t_1^3]-v_1[t_1^2|t_1^2]+v_1[t_1^3|t_1]
-3v_1[t_1|t_2]+2[t_1|t_1t_2]+2[t_1^2|t_1^3]-2[t_1^2|t_2]+2[t_1t_2|t_1]$.
\end{enumerate}
\end{lem}

\begin{proof}
Direct calculation.
\end{proof}

\begin{cor}\label{permcycles}
The following elements are cocycles in the cobar construction $\Omega^*(P;Q)$:
\begin{enumerate}
\item $[\ ]$ and $[\zeta_1]$;
\item $q_1^2[\zeta_1]+q_0q_1[\zeta_1^2]+q_0^2[\zeta_1^3]$;
\item $q_2[\zeta_1|\zeta_1]+q_1[\zeta_1|\zeta_1^3]+q_1[\zeta_1^2|\zeta_1^2]+
q_1[\zeta_1^3|\zeta_1]+q_1[\zeta_1|\zeta_2]+q_0[\zeta_1|\zeta_1\zeta_2]+
q_0[\zeta_1^2|\zeta_1^3]+q_0[\zeta_1^2|\zeta_2]+q_0[\zeta_1\zeta_2|\zeta_1]$.
\end{enumerate}
Moreover, these elements define the classes $1$, $h_0$,
$\langle h_1,q_0^2,h_0\rangle$ and $\langle h_0,q_0,h_1^2\rangle$ in $H^*(P;Q)$ and they are permanent cycles in the algebraic Novikov spectral sequence.
\end{cor}

\begin{proof}
Checking the first Massey product is straightforward. 
The second stops being  difficult once one realizes that 
$d\left(q_2[\zeta_1]+q_1[\zeta_2+\zeta_1^3]+q_0[\zeta_1\zeta_2]\right)=
q_0[\zeta_1^2|\zeta_1^2]$.
\end{proof}

The main result of this section is the following proposition, which completely describes the localized algebraic Novikov spectral sequence.

\begin{prop}\label{localg.nss}
In the localized algebraic Novikov spectral sequence 
\[
h_0^{-1}H^*(P;Q)=\F_2[h_0^{\pm 1},q_1^2,q_2,q_3,\ldots]
\implies \oa_1^{-1}E_2(\STop;BP)\,,
\]
the elements $1$, $h_0$, $q_1^2$ and $q_2$ are permanent cycles, while $d_1q_{n+1}=q_n^2h_0$ for $n\geq 2$.
\end{prop}

\begin{proof}
The images in $\Omega^*(E[\zeta_1];Q)$ of the elements in Corollary \ref{permcycles} are $[\ ]$, $[\zeta_1]$, $q_1^2[\zeta_1]$, $q_2[\zeta_1|\zeta_1]$, respectively. 
So there are permanent cycles in the algebraic Novikov $E_2$-term that map to $1$, $h_0$, $q_1^2h_0$, and $q_2h_0^2$ in the localized $E_2$-term. We are left with proving the differential, so suppose that $n\geq 2$ and write $E$ for $E[\zeta_1]$.

By Proposition \ref{Ploc}, there exists a positive integer $N$ such that $q_{n+1}h_0^N$ is in the image of $H^*(P;Q^1)\rightarrow H^*(E;Q^1)$. Pick an $X$ in $H^{N,*}(P;Q^1)$ mapping to $q_{n+1}h_0^N$. To complete the proof of the proposition it is enough to calculate $d_1X$ in the unlocalized algebraic Novikov spectral and check that its image under $H^*(P;Q)\rightarrow H^*(E;Q)$ is $q_n^2h_0^{N+1}$.

Since $\Omega^*(P;Q)\rightarrow\Omega^*(E;Q)$ is surjective, we can find a cocycle in $\Omega^*(P;Q)$ that represents $X$ and maps to $q_{n+1}[\zeta_1]^N$ in $\Omega^*(E;Q)$. Now all elements of the monomial basis for $\Omega^*(P;Q)$ that include a tensor factor containing some monomial in $P$ other than $\zeta_1$ map to zero in $\Omega^*(E;Q)$. This means that when we write our cocycle in this monomial basis it must contain the term $q_{n+1}[\zeta_1]^N$. We write the cocyle representing $X$ as $q_{n+1}[\zeta_1]^N+x$, where $x$ is a linear combination of other basis elements.

By \eqref{assgraded} we have a surjection
\[
I\Omega^*(BP_*BP)\to\Omega^*(P;Q^1)\,.
\]
We will make use of the set-theoretic splitting that in each term of a linear combination of monomial basis elements replaces each $\zeta_i$ by $t_i$ and each $q_i$ by $v_i$. (Remember that $v_0=2$; this splitting is not linear.) With this choice of splitting, $v_{n+1}[t_1]^N+y$ is selected to map to our cocycle representing $X$, where $y$ is a linear combination of terms, each of which involves, as a tensor factor, some monomial in the $t_i$'s other than the monomial $t_1$, and such that each nonzero coefficient is $v_i$ for some $i$.

Since $q_{n+1}[\zeta_1]^N+x\in\Omega^*(P;Q)$ is a cocycle, $d(v_{n+1}[t_1]^N+y)\in I^2\Omega^*(BP_*BP)$. Mapping to 
$\text{gr}^2\Omega^*(BP_*BP)=\Omega^*(P;Q^2)$ gives an element representing $d_1X\in H^*(P;Q)$. As explained at the start of the proof, we wish to understand the image of this element in $H^*(E;Q)$. 

To do this we will consider the $BP_*$-basis of the cobar construction given by placing a monomial in the $t_i$'s in each tensor factor. Any element of $I^2\Omega^*(BP_*BP)$ is uniquely a linear combination of these elements with coefficients in $I^2$. Of these terms, only those of the form $\alpha[t_1]^j$ with $\alpha\notin I^3$ map nontrivially to $\Omega^*(E;Q^2)$. The elements 
$d(v_{n+1}[t_1]^N)$ and $dy$ are linear combinations of these basis elements with coefficients in $BP_*$. Since $q_{n+1}[\zeta_1]^N$ is not a cocycle, neither set of coefficients by themselves need to lie in $I^2$, though their sums do. First, we look at the contribution from $d(v_{n+1}[t_1]^N)$.

\begin{lem}
For $n\geq 1$, the coefficient of $[t_1]^{N+1}$ in $d(v_{n+1}[t_1]^N)$ is $v_n^2$ mod $I^3$.
\end{lem}

\begin{proof}
Because $t_1$ is primitive it is enough to investigate the coefficient of $t_1$ in $\eta_Rv_{n+1}$. Since the elements $\eta_Rv_{n+1}$ and $v_{n+1}=\eta_Lv_{n+1}$ have the same augmentation, we have 
\[
\eta_R(v_{n+1})\equiv v_{n+1}+ct_1\quad\text{mod }(t_1^2,t_2,t_3,\ldots)
\] 
for some $c\in BP_{2(2^{n+1}-2)}$. The only monomial in the $v_i$'s of the degree of $c$ that is not in $I^3$ is $v_n^2$. Moreover, $2v_n^2t_1\in I^3$ so that
\[
\eta_R(v_{n+1})\equiv v_{n+1}+bv_n^2t_1\quad\text{mod }I^3+(t_1^2,t_2,t_3,\ldots)
\]
where $b=0$ or $1$. Since \cite[$5.1$]{miller1976novikov}
\[
\eta_R(v_{n+1})\equiv v_{n+1}+v_nt_1^{2^n}-v_n^2t_1
\quad\text{mod }(2,v_1,\ldots,v_{n-1})
\]
we must have $b=1$.
\end{proof}

Mapping $v_n^2[t_1]^{N+1}\in I^2\Omega^*(BP_*BP)$ to $\text{gr}^2\Omega^*(BP_*BP)=\Omega^*(P;Q^2)$ gives $q_n^2[\zeta_1]^{N+1}$. Mapping further to $\Omega^*(E;Q)$, gives a cocycle representing $q_n^2h_0^{N+1}$. In order to complete the proof it suffices to show that the coefficient of $[t_1]^{N+1}$ in $dy$ 
is zero.

Recall that $y$ is a linear combination of terms, each of which involves, as a tensor factor, some monomial in the $t_i$'s other than the monomial $t_1$. The differential in the cobar complex makes use of the right unit and the diagonal map in $BP_*BP$. When evaluating $dy$, the right unit is used on the coefficients of the terms in $y$. Since none of the monomials occurring in $y$ are of the form $[t_1]^N$, $BP_*$ multiples of $[t_1]^{N+1}$ cannot arise from this part of the differential. Thus we just need to consider terms coming from the diagonal map. The following simple lemma is crucial.

\begin{lem}\label{t_1tensort_1}
The only monomials in the $t_i$'s that contain a nonzero $BP_*$-multiple of $t_1\otimes t_1$ in their diagonal are $t_1^2$ and $t_2$.
\end{lem}

\begin{proof}
Recall that we have an inclusion $BP_*\subset H_*(BP)=\Z_{(2)}[m_1,m_2,\ldots]$ given by the Hurewicz homomorphism. Thus, we can compute in $H_*(BP\wedge BP)=H_*(BP)[t_1,t_2,\ldots]$ and there we have (see \cite{miller1976novikov}) an inductive formula for the diagonal of $t_n$.
\[
\Delta t_n=
\sum_{i+j+k=n}m_it_j^{2^i}\otimes t_k^{2^{i+j}}-
\sum_{i=1}^nm_i(\Delta t_{n-i})^{2^i}\,.
\]
The first sum does not contain a term $t_1\otimes t_1$. Moreover, the only terms in $\Delta t_{n-i}$ with a $1$ on one side or the other are $t_{n-i}\otimes 1$ and $1\otimes t_{n-i}$. Thus, in the expression of $(\Delta t_{n-i})^{2^i}$, the only way one can achieve $t_1\otimes t_1$ is with $i=1$ and $n-i=1$ so that $n=2$.

The diagonal is multiplicative and so one can achieve $t_1\otimes t_1$ in the diagonal of a monomial only in the cases $t_2$ and $t_1^2$.
\end{proof}

Now consider a tensor product of monomials, a basis element in the cobar construction. The differential is computed by applying the reduced diagonal to each factor and taking the alternating sum. One receives a term that is a $BP_*$-multiple of $[t_1]^{N+1}$ only by starting with a tensor product of monomials in which all but one term is $t_1$, and the remaining term is either $t_1^2$ or $t_2$. 

We should call attention to a subtlety here. When the reduced diagonal is applied to a monomial, the result is a $BP_*$-linear combination of monomials. Given a basis element of the cobar complex, to express the value of the differential on it as a $BP_*$-linear combination of tensor products of such monomials, one needs to pull coefficients outside the tensor products. This operation is nontrivial since the tensor products, while formed over $BP_*$, use the left and right actions on the right and left factors, respectively. In particular, $t\otimes vt'=\eta_R(v)t\otimes t'$. The element $\eta_R(v)$ will itself be a linear combination of monomials in the $t_i$'s (where we now include $1$ as $t_0$) so if the expression involves more than $[t_1]$'s before this maneuver, it will continue to involve more than $[t_1]$'s afterwards as well.

Now recall that $y$ has internal dimension $2(2^{n+1}-1)+2N$. The internal dimensions of 
$[t_1]^{N-i}[t_1^2][t_1]^{i-1}$ and $[t_1]^{N-i}[t_2][t_1]^{i-1}$ 
are $2(N+1)$ and $2(N+2)$, respectively and so the coefficients of these basis elements in $y$ must have internal dimensions $2(2^{n+1}-2)$ and $2(2^{n+1}-3)$, respectively. But recall that the coefficient of each term appearing in $y$ is a $v_i$. The first dimension does not occur as the dimension of a $v_i$, and the second occurs only for $n=1$. 

We note that when $n=1$ such terms do occur, as we see in Lemma \ref{BPcocycles}, where the third cocycle contains the terms $v_2[t_1|t_1]$ and 
$-3v_1[t_1|t_2]$. These provide two canceling $v_1^2[t_1|t_1|t_1]$ terms. Of course, this is how we saw $q_2h_0^2$ was a permanent cycle.

But now we are assuming $n\geq 2$, so the proof of Proposition \ref{localg.nss}
is complete.
\end{proof}

We see immediately from the proposition that the $E_2$-page of the localized algebraic Novikov spectral sequence consists of permanent cycles and so we obtain the following corollary.

\begin{cor}
The $E_{\infty}$-page of the localized algebraic Novikov spectral sequence is \[\F_2[h_0^{\pm 1},q_1^2,q_2]/(q_2^2).\]
\end{cor}

\section{What happens to $\oa_s$?}
\label{secalphas}

We now return to the elements $\oa_s$ of Section \ref{secane2}. In order to say what happens to them under the localization map it is convenient to consider the mod $2$ Moore spectrum analogues of our results for $\STop$, which are of interest in their own right. We write $\STop/2$ for the mod $2$ Moore spectrum. Just as before, when we filter 
$\Omega^*(BP_*BP;BP_*(\STop/2))=\Omega^*(BP_*BP/2)$ by powers of the
augmentation ideal $BP_*BP/2\to\F_2$, we arrive at the 
``mod 2 algebraic Novikov spectral sequence'' 
\[
E_1^{s,t,u}=H^{s,u}(P;[Q/(q_0)]^t)
{\implies} 
H^{s,u}(BP_*BP/2)\,,\quad 
d_r:E_r^{s,t,u}\to E_r^{s+1,t+r,u}\,.
\]
Again $h_0\in E_1^{1,0,2}$ is a permanent cycle, and,
continuing to follow the argument above, we arrive at the following result.

\begin{prop}
The localized mod 2 algebraic Novikov spectral sequence converges, and has
\[
E_1=h_0^{-1}H^*(P;Q/(q_0))=\F_2[h_0^{\pm 1},q_1,q_2,\ldots]\,.
\]
The elements $q_1$ and $q_2$ are permanent cycles in this spectral sequence and we have 
\[
d_1q_{n+1}=q_n^2h_0\, \hbox{ for } n\geq 2\,.
\]
The map $\STop\to\STop/2$ induces a map between the localized algebraic Novikov spectral sequences. At the $E_1$-page, the map is given by the inclusion
\[
\F_2[h_0^{\pm 1},q_1^2,q_2,q_3,\ldots]\to\F_2[h_0^{\pm 1},q_1,q_2,\ldots]\,.
\]
At the $E_{\infty}$-page, it is given by the inclusion 
\[
\F_2[h_0^{\pm 1},q_1^2,q_2]/(q_2^2)\to \F_2[h_0^{\pm 1},q_1,q_2]/(q_2^2)\,.
\]
\end{prop}

We now locate the elements $\oa_s$ in the localized algebraic Novikov
spectral sequence.

\begin{lem}\label{alpha_1-free}
For $s\neq 2$, $\oa_s$ is $\oa_1$-free and its image in $\oa_1^{-1}H^*(BP_*BP)$ is detected by $q_1^{s-1}h_0$ when $s$ is odd and by $q_1^{s-4}q_2h_0$ when $s$ is even.
\end{lem}

\begin{proof}
When $s$ is odd $\oa_s$ has a cocycle representative with leading term $sv_1^{s-1}[t_1]$ because $\eta_Rv_1=v_1+2t_1$. All other terms have the same filtration and involve higher powers of $t_1$. Thus, $\oa_s$ is detected by $q_1^{s-1}h_0$. 

Suppose $s$ is even and bigger than 2. Because the map induced by $\STop\rightarrow\STop/2$ between our localized algebraic Novikov spectral sequences is injective at each page, it suffices to check the result for the image of $\overline{\alpha}_s$ in $H^*(BP_*BP;BP_*/2)$. For $s>4$ one can find (see \cite[$4.4.35$]{ravenel2004complex}, for instance) an explicit cocycle representative for 
$\oa_s$:
\[
v_1^{s-4}v_2[t_1]+v_1^{s-3}[t_2]+v_1^{s-3}[t_1^3]\,.
\]
This element is detected by $q_1^{s-4}q_2h_0$ in the localized algebraic Novikov spectral sequence. 

For $s=4$ one finds, by direct computation, the following cocycle respresentative for $\oa_4$:
\[
[t_1^4]+v_2[t_1]+v_1[t_2]+v_1[t_1^3]+v_1^2[t_1^2]\,.
\]
Upon multiplying by $\oa_1^3$ we have leading term $[t_1^4|t_1|t_1|t_1]$. It is classical that $h_0^3h_2=0$ in $H^*(P)$. Find $y\in\Omega^3P$ with $dy=[\zeta_1^4|\zeta_1|\zeta_1|\zeta_1]$. Then obtain $y'\in\Omega^3(BP_*BP)$ by replacing $\zeta$'s by $t$'s. Using Lemma \ref{t_1tensort_1} we see that $dy'$ cannot contain $v_2[t_1|t_1|t_1|t_1]$. Thus, picking off the elements of filtration $1$ with only single powers of $t_1$'s appearing in
\[\bigg([t_1^4]+v_2[t_1]+v_1[t_2]+v_1[t_1^3]+v_1^2[t_1^2]\bigg)
\cdot[t_1|t_1|t_1]+dy'\]
gives $v_2[t_1|t_1|t_1|t_1]$. We deduce that $\oa_1^3\oa_4$ is detected by $q_2h_0^4$ in the localized algebraic Novikov spectral sequence.
\end{proof}

We can now obtain an explicit description of the localized Adams-Novikov $E_2$-page, in terms of elements which exist before localizing.

\begin{cor}
$\oa_1^{-1}H^*(BP_*BP)=\F_2[\oa_1^{\pm 1},\oa_3,\oa_4]/(\oa_4^2)$.
\end{cor}


\begin{proof}
Consider the natural map 
$\Z_{(2)}[\oa_1,\oa_3,\oa_4]
\to H^*(BP_*BP)$. It can be checked that 
$\oa_1\oa_4^2=0$ in $H^*(BP_*BP)$; this follows
for example from Toda's relation $\eta\sigma^2=0$ \cite{toda1963composition} 
in $\pi_*(\STop)$.
We also have $2\oa_1=0$. This map thus factors through a map
\[
\Z_{(2)}[\oa_1,\oa_3,\oa_4]/(2\oa_1,\oa_1\oa_4^2)\to H^*(BP_*BP)\,.
\] 
Inverting $\oa_1$ gives a map 
$f:\F_2[\oa_1^{\pm 1},\oa_3,\oa_4]/
(\oa_1\oa_4^2)
\to\oa_1^{-1}H^*(BP_*BP)$.

We have shown that $\oa_3$ and $\oa_4$ have images in $\oa_1^{-1}H^*(BP_*BP)$ detected by $q_1^2h_0$ and $q_2h_0$, respectively. 
The $E_{\infty}$-page of the localized algebraic Novikov spectral sequence is $\F_2[h_0^{\pm 1},q_1^2,q_2]/(q_2)$; for each bidegree $(s,u)$ there is only one $t$ such that $E_{\infty}^{s,t,u}\neq 0$ and so the filtration is locally finite. 
These facts, together with convergence of the localized algebraic Novikov spectral sequence, allow one to check that $f$ is injective and surjective.
\end{proof}

\section{The localized motivic Adams-Novikov spectral sequence}
\label{secmot}

Here is our main result about motivic homotopy groups. We are working over an
algebraically closed base field of characteristic zero. 

\begin{thm}
\label{motivictheorem}
Let $\eta\in\pi_{1,1}(\SMot)$ and $\sigma\in\pi_{7,4}(\SMot)$ 
denote the elements of motivic
Hopf invariant 1, and $\mu_9$ the nonzero element in $\pi_{9,5}(\SMot)$. Then 
$2\eta=0$, $\sigma^2$ is $\eta$-torsion, and 
\[
\pi_{*,*}(\eta^{-1}\SMot)=\eta^{-1}\pi_{*,*}(\SMot)=
\F_2[\eta^{\pm 1},\sigma,\mu_9]/(\sigma^2)\,.
\]
\end{thm}

The first step is to move to the $2$-complete context. We owe the following
observation to the referee. 

\begin{lem}For any motivic spectrum, 
\[
\eta^{-1}X\to\eta^{-1}(X^\wedge_2)
\]
is an equivalence.  
\label{eta-local-completion}
\end{lem}
\begin{proof}
Since $2\eta=0$, there is a map 
$\eta^{-1}X/2\to\eta^{-1}X$, natural in the motivic spectrum $X$, that
splits the map induced by $X\to X/2$. On the other hand, the completion
map $X/2\to(X/2)^\wedge_2$ is an equivalence. The result follows. 
\end{proof}

We now briefly recall the theorems of Voevodsky 
\cite{hu2011remarks,voevodsky2003reduced} concerning mod $2$ motivic homology 
and the motivic Steenrod algebra over an algebraically closed field of 
characteristic $0$. 

Motivic homotopy and homology are graded by a free abelian group of rank two.
One pair of parameters is \emph{dimension} and \emph{weight} and with this 
bigrading $|\tau|=(0,-1)$. Other parameters are the \emph{coweight} 
(or ``Milnor-Witt degree'' \cite{dugger2015low}) and the Novikov or 
Chow \cite{dugger2013motivic} \emph{degree}; these satisfy the relations
\begin{center}
cowt $+$ wt = dim $\,,\quad$ cowt $-$ wt = deg\,.
\end{center}
Each of these parameters has its uses. For example, Morel showed that the 
motivic stable homotopy ring is zero in negative coweight, and given by the
Milnor-Witt $K$-theory (which he defined for the purpose) 
in coweight zero. Our bigrading will be by dimension and weight. 

The coefficient ring of mod 2 motivic homology, written $H$,
is $\mathbb{M}_2=\F_2[\tau]$. 

Hu, Kriz, and Ormsby \cite{hu2011remarks} (see also \cite{dugger2010motivic}) 
describe a ``motivic Adams-Novikov 
spectral sequence.'' To circumvent the fact that the full structure 
of the motivic Thom spectrum $MGL$ is unknown they work with the
$H$-completion. They show that 
this motivic spectrum splits as a wedge of suspensions of a motivic analogue
of the Brown-Peterson spectrum, denoted here by $BPM$, and use it to
construct a spectral sequence.

They show that the $H$-complete motivic analogue of the Hopf algebroid $(BP_*,BP_*BP)$ is simply the classical one tensored with $\Z_2[\tau]$ (where $\ZZ_2$ denotes the 2-adic integers). It follows that the $E_2$-page of the motivic Adams-Novikov spectral sequence is obtained from the classical one by completing at 2 and adjoining $\tau$, and that the corresponding algebraic Novikov spectral sequence is obtained by adjoining $\tau$. Thus, our work above has the following consequence.

\begin{cor}
Over an algebraically closed field of characteristic zero, the $H$-complete 
motivic Adams-Novikov $E_2$-page localizes to 
$\overline{\alpha}_1^{-1}E_2=
\F_2[\tau,\overline{\alpha}_1^{\pm 1},\overline{\alpha}_3,\overline{\alpha}_4]/
(\overline{\alpha}_4^2)$.
\end{cor}

Dugger and Isaksen observe that the motivic Adams-Novikov spectral sequence
converges to $\pi_*((\SMot)^{\wedge}_{H})$. 
In \cite{hu2011convergence} this mod 2 homology completion is 
identified with the 2-adic completion $(\SMot{})^\wedge_2$; 
so the $H$-completed motivic Adams-Novikov spectral sequence has the form
\[
E_2^{s,u,w}=H^*(BP_*BP)\otimes\ZZ_2[\tau]^{s,u,w}
\overset{s}{\implies}
\pi_{u-s,w}((\SMot)^{\wedge}_2)\,,\quad
d_r:E_r^{s,u,w}\to E_r^{s+r,u+r-1,w}\,.
\]
If $x\in H^{s,u}(BP_*BP)$ is nonzero then $u$ is even and $\tau^nx$ defines an element of $H^*(BP_*BP)[\tau]^{s,u,u/2-n}$. We can recover the classical Adams-Novikov spectral sequence by forgetting the weight and setting $\tau=1$.

Inverting $\eta$ does not harm convergence of this spectral sequence, since
powers of $\oa_1$ constitute the vanishing line at $E_2$; 
it converges to $\eta^{-1}\pi_*((\SMot)^\wedge_2)$. 

The classical differential $d_3\overline{\alpha}_3=\overline{\alpha}_1^4$ 
appears motivically as 
\[
d_3\overline{\alpha}_3=\tau\overline{\alpha}_1^4\,.
\]
When we invert $\oa_1$, this differential has the effect of killing $\tau$, 
and using the relations above we find that 
\[
E_\infty=\F_2[\overline{\alpha}_1^{\pm 1},\overline{\alpha}_3^2,\overline{\alpha}_4]/
(\overline{\alpha}_1\overline{\alpha}_4^2)=
\F_2[\overline{\alpha}_1^{\pm 1},\overline{\alpha}_4,\overline{\alpha}_5]/
(\overline{\alpha}_4^2)\,.
\]

To see what this implies about motivic homotopy groups, note that 
classically $\eta$, $\sigma$, $\mu_9$ and $\eta\mu_9=\mu_{10}$ are detected by $\overline{\alpha}_1$, $\overline{\alpha}_4$, $\overline{\alpha}_5$ and $\overline{\alpha}_1\overline{\alpha}_5=\overline{\alpha}_3^2$, respectively, in the Adams-Novikov spectral sequence. These facts hold motivically as well
\cite{hornbostel}, and
the relation $\eta\sigma^2=0$ is true motivically also in the 2-complete sphere
\cite{isaksen2014stable}. So we receive a map 
\[
\F_2[\eta,\sigma,\mu_9]/(2\eta,\eta\sigma^2)\to\pi_{*,*}((\SMot))^\wedge_2\,.
\]

Now $\eta$, $\sigma$, and $\mu_9$ are detected by $\oa_1$, $\oa_4$, and
$\oa_5$ respectively. Convergence of the $\eta$-localized $H$-completed
motivic Adams-Novikov spectral sequence then shows that 
\[
\F_2[\eta^{\pm1},\sigma,\mu_9]/(2\eta,\eta\sigma^2)\to
\pi_{*,*}((\SMot)^\wedge_2)
\]
is an isomorphism.

This completes the proof of Theorem \ref{motivictheorem}.

\numberwithin{equation}{subsection}
\section{A comparison of spectral sequences}
\subsection{The diagram}
In this final section, we complete the calculation of a square of spectral sequences, a localized version of the following square.
\[\xymatrixcolsep{100pt}\xymatrix{
H^*(P;Q)[\tau]\ar@{=>}[r]^{\text{CESS}}\ar@{=>}[d]_-{\text{ANSS}[\tau]}& 
E_2(\SMot;H)\ar@{=>}[d]^{\text{MASS}}\\
E_2(\SMot;BPM)\ar@{=>}[r]^{\text{MNSS}}& \pi_{*,*}(\SMot)
}\]
The right spectral sequence is the motivic Adams spectral sequence as studied in \cite{dugger2010motivic,hu2011remarks}. The bottom spectral sequence is the motivic Adams-Novikov spectral sequence described above, which was first studied in \cite{hu2011remarks}. The left spectral sequence is the \emph{motivic algebraic Novikov spectral sequence}, obtained by filtering $\pi_*(BPM)=BP_*[\tau]$ by powers of the kernel of the augmentation $\pi_*(BPM)\to\F_2[\tau]$. By the results of Hu, Kriz, and Ormsby \cite{hu2011remarks} this is simply the algebraic Novikov spectral sequence described in Section 
\ref{secanss} extended by adjoining $\tau$. The grading of $H^*(P;Q)[\tau]$ follows that of $H^*(BP_*BP)[\tau]$. If $x\in H^{s,u}(P;Q^t)$ is nonzero then $u$ is even and $\tau^nx$ defines an element of $H^*(P;Q)[\tau]^{s,t,u,u/2-n}$. The top spectral sequence is the Cartan-Eilenberg spectral sequence associated to the extension of Hopf algebras
\begin{align}\label{extHopf}
\mathbb{M}_2\otimes P\to A_{\text{Mot}}\to\mathbb{M}_2\otimes E.
\end{align}
This motivic Cartan-Eilenberg spectral sequence is indexed just as in \eqref{CESS}, but with the additional weight grading that is preserved by differentials. The vanishing lines of \eqref{vanishingAdamsNov} and \eqref{vanishingAdams} ensure that we can localize all the spectral sequences to obtain a square of convergent spectral sequences. The behavior of these spectral sequences is summarized in the following diagram.
\[\xymatrixcolsep{100pt}\xymatrix{
\F_2[\tau,h_0^{\pm 1},q_1^2,q_2,q_3,\ldots]\ar@{=>}[r]^{d_3q_1^2=\tau h_0^3}
\ar@{=>}[d]_-{d_1q_{n+1}=q_n^2h_0,\ n\geq 2}& 
\F_2[h_0^{\pm 1},v_1^4,v_2,v_3,\ldots]
\ar@{=>}[d]^-{d_2v_{n+1}\equiv v_n^2h_0,\ n\geq 2}\\
\F_2[\tau,\overline{\alpha}_1^{\pm 1},\overline{\alpha}_3,\overline{\alpha}_4]/
(\overline{\alpha}_1\overline{\alpha}_4^2)
\ar@{=>}[r]^{d_3\overline{\alpha}_3=\tau\overline{\alpha}_1^4}& 
\F_2[\eta^{\pm 1},\sigma,\mu_9]/(\eta\sigma^2)
}\]
We have calculated the left spectral sequence and the bottom one in the earlier sections of this paper. Guillou and Isaksen calculated the $E_2$-page of the localized motivic Adams spectral sequence in \cite{guillou2014eta}. In the next section, we will give a different proof of their result by calculating the top spectral sequence. In the final section, we will use the techniques of \cite{miller1981relations} to determine the differentials in the localized motivic Adams spectral sequence, verifying another conjecture of Guillou and Isaksen \cite{guillou2014eta}.

\subsection{The localized Cartan-Eilenberg spectral sequence}
The extension of Hopf algebras \eqref{extHopf} gives rise to a Cartan-Eilenberg spectral sequence, which we may localize by inverting 
$h_0\in H(P;Q)[\tau]^{1,0,2,1}$.

\begin{lem}
In the localized Cartan-Eilenberg spectral sequence we have 
$d_3q_1^2=\tau h_0^3$. The classes $q_1^4$ and $q_n$ for $n\geq 2$ are permanent cycles and so 
\[E_{\infty}=\F_2[h_0^{\pm 1},q_1^4,q_2,q_3,\ldots].\]
\end{lem}

\begin{proof}
The differential $d_3q_1^2=\tau h_0^3$ follows from the unlocalized differential $d_3\langle h_1,q_0^2,h_0\rangle=\tau h_0^4$ and this is forced on us by our limited knowledge of $H^*(A_{\text{Mot}})$. Degree considerations show that $q_1^4$, and $q_n$ for $n\geq 2$, are permanent cycles.
\end{proof}

We can now prove the following result, established also by 
Guillou and Isaksen \cite{guillou2014eta}. We note that they follow the classical conventions at $p=2$ and denote by $h_1$ the class that we call 
$h_0\in E_2^{1,2,1}(\SMot;H)$.

\begin{cor}
There exist classes $v_1^4,v_2,\ldots$, with $|v_1^4|=(0,8,4), |v_n|=(0,2(2^n-1),2^n-1)$, such that
\[
h_0^{-1}E_2(\SMot;H)=\F_2[h_0^{\pm 1},v_1^4,v_2,v_3,\ldots]\,.
\]

\end{cor}

\begin{proof}
We choose a representative for $q_1^4$, which we call $v_1^4$, and for 
$n\geq 2$ we choose representatives for $q_n$, which we call $v_n$. Since the associated graded algebra is free on the classes of these generators, the result follows.
\end{proof}

\subsection{Comparing Adams spectral sequences}

In this section we will complete the calculation of the localized motivic 
Adams spectral sequence. 
By finding representatives, one sees that in the localized motivic Adams spectral sequence for the $\eta$-local sphere spectrum the elements $v_1^4$ and $v_2$ are permanent cycles. 
For the other generators, we have the following proposition, which follows from the techniques of \cite{miller1981relations}.

\begin{prop}\label{MASSdiff}
For $n\geq 2$, we have $d_2v_{n+1}=v_n^2h_0$ modulo 
higher Cartan-Eilenberg filtration.
\end{prop}

We will give an improvement, due to the first author, of the statement and the proof of the comparison result of \cite{miller1981relations} (which, in turn, followed ideas from \cite{novikov1967methods}).
The second author is eager to use this opportunity to clarify the proof given in \cite{miller1981relations}, and to fill a gap: Lemma $6.7$ is not correct as stated there. 
What follows is a correct statement that serves the purpose in \cite{miller1981relations}, and which will be used in the proof presented here as well. 
This lemma relates to the comparison of two boundary maps, and its importance cannot be overstated. 
It deals with the following situation.
Let $A\to B\to C$ and $X\to Y\to Z$ be cofiber sequences. Smash them together
to form the following commutative diagram of cofiber sequences.
\[
\xymatrix{
A\wedge X \ar[d] \ar[r] & A\wedge Y \ar[d] \ar[r] & A\wedge Z \ar[d] \\
B\wedge X \ar[d] \ar[r] & B\wedge Y \ar[d] \ar[r] & B\wedge Z \ar[d] \\
C\wedge X        \ar[r] & C\wedge Y        \ar[r] & C\wedge Z
}\]
Let $b$ be an element of $\pi_n(B\wedge Y)$ that maps to $0$ in 
$\pi_n(C\wedge Z)$. 
Then there is an element $a\in\pi_n(A\wedge Z)$ mapping to the image of $b$ in $\pi_n(B\wedge Z)$, and an element $c\in\pi_n(C\wedge X)$ mapping to the image
of $b$ in $\pi_n(C\wedge Y)$. 

\begin{lem}[May \cite{may2001additivity}]\label{lemma-may}
The elements $a$ and $c$ can be chosen so that they have the same 
image (up to a conventional sign) in $\pi_{n-1}(A\wedge X)$ under the 
boundary maps associated to the 
cofiber sequences along the top and the left edge of the diagram.
\end{lem}

This statement is a small part of an elaborate structure enriching the 
displayed $3\times 3$ diagram. 
This structure is described in detail and proved by May in \cite{may2001additivity}. 
In the founding days of the theory of triangulated categories, Verdier \cite{deligne1983faisceaux} showed that a $2\times2$ diagram can always be extended to a $3\times3$ diagram of cofiber sequences. 
An analysis of his proof reveals that it actually produces precisely the structure verified by May for the specific case in which the $3\times3$ diagram occurs by smashing together two cofiber sequences. 

For clarity, we will work in the non-motivic context, and in the specific case of $BP$ and $H\F_p$ (for any prime $p$) and the sphere spectrum. 
We will then indicate the general setting under which the result holds and this will prove the proposition just stated. 
Write $H$ for the mod $p$ Eilenberg Mac Lane spectrum. 

So we have the following square of spectral sequences. 
\[\xymatrixrowsep{38pt}\xymatrix{
H^{s,u}(P;Q^t) \ar@{=>}[rr]^{\text{CESS}}
\ar@{=>}[d]_{\text{ANSS}} 
&& E_2^{s+t,u+t}(\STop;H)\ar@{=>}[d]^{\text{ASS}} \\
E_2^{s,u}(\STop;BP)\ar@{=>}[rr]^{\text{NSS}}&& \pi_{u-s}(\STop)
}\]
The initial two are
the algebraic Novikov spectral sequence \eqref{alg.nss} and the Cartan-Eilenberg spectral sequence \eqref{CESS}; the final two are the ($H$-based) Adam 
spectral sequence and the Novikov or $BP$-based Adams, spectral sequence.
Write $d^H_r$ for the Adams differentials, and $d^{AN}_r$ for the differentials
in the algebraic Novikov spectral sequence.

\begin{thm}
Suppose $x\in F^s_{\text{CE}}E_2^{s+t,u+t}(\STop;H)$.
Then the Cartan-Eilenberg filtration of $d_2^{\text{H}}x$ is higher:
\[
d_2^{\text{H}}x\in F^{s+1}_{\text{CE}}E_2^{s+t+2,u+t+1}(\STop;H)\,.
\] 
Moreover, if $x$ is detected in the Cartan-Eilenberg spectral sequence
by $a\in H^{s,u}(P;Q^t)$ then
$d_2^{\text{H}}x$ is detected by 
$d_1^{\text{AN}}a\in H^{s+1,u}(P;Q^{t+1})$.
\label{thmcomparison}
\end{thm}

\begin{proof}
The proof depends upon geometric constructions of the two algebraically defined spectral sequences. Both arise from the canonical $BP$-resolution of $\STop$ and so we recall how this resolution is constructed. From the unit map of the ring spectrum $BP$ we can construct a cofiber sequence
\begin{align}\label{unitBP}
\STop\to BP\to\overline{BP}.
\end{align}
Smashing this cofiber sequence with various smash-powers of $\overline{BP}$ gives the canonical $BP$-resolution of $\STop$. 
\be
\begin{aligned}
\xymatrixcolsep{40pt}\xymatrix@L=-4pt{
\STop\ar[d]& \overline{BP}\ar[d]\ar[l]_-{|}& 
\cdots\ar[l]_-{|}& 
\overline{BP}{}^{\wedge s}\ar[d]\ar[l]_-{|}& 
\overline{BP}{}^{\wedge (s+1)}\ar[d]\ar[l]_-{|}& 
\cdots\ar[l]_-{|}\\
BP^{[0]}& BP^{[1]}&& BP^{[s]}& BP^{[s+1]}
}
\label{canonicalBP}
\end{aligned}
\ee
Here we use the notation
\[
BP^{[s]}=\overline{BP}{}^{\wedge s}\wedge BP,
\]
and the marked arrows indicate that they map from a desuspension. 
The Adams-Novikov spectral sequence for the homotopy of a spectrum $X$ is 
associated
to the exact couple arising by smashing this resolution with $X$ and taking
homotopy groups. The maps in the exact couple will be denoted
\be
\begin{aligned}
i_{BP}:\pi_{u+s}(X\wedge\oBP^{\wedge(s+1)})\to &
\pi_{u+s-1}(X\wedge\oBP^{\wedge s}) \\
j_{BP}:\pi_{u+s}(X\wedge\oBP^{\wedge s})\to & 
\pi_{u+s}(X\wedge BP^{[s]}) \\
k_{BP}:\pi_{u+s}(X\wedge BP^{[s]})\to &
\pi_{u+s}(X\wedge\oBP^{\wedge(s+1)})
\label{ec}
\end{aligned}
\ee

The $E_1$ term arising from this exact couple is isomorphic, as a complex, 
to the cobar complex $\Omega^*(BP_*BP)$.
The algebraic Novikov spectral sequence arises by filtering this complex
by powers of the augmentation ideal in $BP_*$. In geometric terms, we are
filtering $\pi_*(\oBP^{\wedge s})$ by the classical Adams filtration. 

Next we set up the Cartan-Eilenberg spectral sequence. Smashing \eqref{unitBP} with a spectrum $X$ and applying mod $2$ homology gives a short exact sequence
\[0\to H_*(X)\to H_*(BP\wedge X)\to H_*(\overline{BP}\wedge X)\to 0\]
(by the K\"unneth formula) and thus a long exact sequence
\be
\cdots\to E_2^{t,u}(X;H)\to E_2^{t,u}(X\wedge BP;H)\to 
E_2^{t,u}(X\wedge\overline{BP};H)
\overset{\delta}{\to} E_2^{t+1,u}(X;H)\to\cdots.
\label{e2les}
\ee
This means that applying $E_2(-;H)$ to \eqref{canonicalBP} gives an exact couple and hence a spectral sequence. We index the spectral sequence so that
\[E_1^{s,t,u}=E_2^{t,u+t}(BP^{[s]};H)
\overset{s}{\implies} E_2^{s+t,u+t}(\STop;H).\]
A change of rings theorem identifies $E_2^{t,u+t}(BP^{[s]};H)$ with $(Q^t\otimes\overline{P}^{\otimes s})_u$, so that our $E_1$-term is isomorphic, as a complex, to $\Omega^*(P;Q)$. Thus $E_2^{s,t,u}=H^{s,u}(P;Q^t)$. This spectral sequence is, in fact, the Cartan-Eilenberg spectral sequence of \eqref{CESS}. The Cartan-Eilenberg filtration of $E_2(\STop;H)$ is given by
\begin{align}\label{CEfilt}
F^s_{\text{CE}}E_2^{s+t,u+t}(\STop;H)=\im\Big(
\delta^s:
E_2^{t,u+t}(\overline{BP}{}^{\wedge s};H)\to 
E_2^{s+t,u+t}(\STop;H)
\Big).
\end{align}
If $x=\delta^sz$ then $x$ is detected in the Cartan-Eilenberg $E_1$-page by the image of $z$ under the map 
$j_{BP}:E_2^{t,u+t}(\oBP^{\wedge s};H)\to E_2^{t,u+t}(BP^{[s]};H)$.

We next recall a construction of the Adams spectral sequence for a spectrum $X$. We have a cofiber sequence $\STop\to H\to\overline{H}$ and we define
\[
H^{[t]}=H\wedge\overline{H}{}^{\wedge t}.
\]
Note that we put the factor of $H$ at the left rather than at the right
as in the case of $BP$; this will let us keep these two types of resolution
separate. The cofiber sequences 
\[\xymatrixcolsep{40pt}\xymatrix@L=-4pt{
\overline{H}{}^{\wedge (t+1)}\wedge X\ar[r]_-{|}& 
\overline{H}{}^{\wedge t}\wedge X\ar[r]& 
H^{[t]}\wedge X}\]
link together as in \eqref{canonicalBP} and the Adams spectral sequence for $X$ is obtained by applying $\pi_*(-)$. We have $E_1^{t,u}=\pi_u(H^{[t]}\wedge X)$ and $E_2^{t,u}=H^{t,u}(A;H_*(X))$. 
The structure maps in the exact couple will be denoted as in \eqref{ec} but 
subscripted with $H$.

The map of ring spectra $BP\to H$ descends uninque to a map 
\[
\delta:\overline{BP}\to\overline{H}
\]
such that $i_H\delta=i_{BP}$.
We will denote this map and all the maps it induces by $\delta$, even if they
involve the swap map $T$. For example, we have compatible maps 
\[
\xymatrix@R=-3pc{
\delta:\oH^{\wedge t}\wedge X\wedge\oBP \ar[r]^{1\wedge T} &
\oH^{\wedge t}\wedge\oBP\wedge X \ar[r]^{1\wedge\delta\wedge1} &
\oH^{\wedge(t+1)}\wedge X \\
\delta:H^{[t]}\wedge X\wedge\oBP \ar[r]^{1\wedge T} & 
H^{[t]}\wedge\oBP\wedge X \ar[r]^{1\wedge\delta\wedge1} & 
H^{[t+1]}\wedge X
}
\]
for any spectrum $X$, and so a map of spectral sequences 
\[
\delta:E_r^{t,u}(X\wedge\oBP;H)\to E_r^{t+1,u}(X;H)\,.
\]
In particular, with
$X=\oBP^{\wedge s}$, we have compatible maps
\begin{align*}
\delta:\oH^{\wedge t}\wedge\oBP^{\wedge(s+1)}\to &
\oH^{\wedge(t+1)}\wedge\oBP^{\wedge s}\\
\delta:H^{[t]}\wedge\oBP^{\wedge(s+1)}\to &
\oH^{[t+1]}\wedge\oBP^{\wedge s}\,. 
\end{align*}
At $E_2$, this map is the boundary map $\delta$ in \eqref{e2les}. 

We now address the first claim of the theorem. Suppose that 
$x\in F^s_{CE}E_2^{s+t,u+t}(\STop;H)$, so that $x=\delta^sz$ for some $z\in E_2^{t,u+t}(\overline{BP}{}^{\wedge s};H)$. Since $\delta$ is a map of spectral sequences, $d_2^{H}x=\delta^s(d_2^{H}z)$. Thus in order to show that 
\[
d_2^{H}x\in F^{s+1}_{CE}E_2^{s+t+2,u+t+1}(\STop;H)
\]
it suffices to find an element mapping to $d_2^{H}z$ under
\[
\delta:E_2^{t+1,u+t+1}(\overline{BP}{}^{\wedge (s+1)};H)\to 
E_2^{t+2,u+t+1}(\overline{BP}{}^{\wedge s};H)\,.
\]

To compute $d_2^{H}z$, begin by picking a representative 
\be
\label{zprime}
z'\in\pi_{u+t}(H^{[t]}\wedge\overline{BP}{}^{\wedge s})
\ee
for $z$.
Since $0=d^H_1z'=j_Hk_Hz'$, $k_Hz'=i_Hy'$ for some 
$y'\in\pi_{u+t+1}(\oH^{\wedge(t+2)}\wedge\oBP^{\wedge s})$. Then 
$j_Hy'\in\pi_{u+t+1}(H^{[t+2]}\wedge\oBP^{\wedge s})$ is a cocycle
representing $d^H_2z$. These maps arise by applying $\pi_*$ to the bottom 
row in the following diagram.
\be
\xymatrix{
& & 
\oH^{\wedge(t+1)}\wedge\oBP^{\wedge(s+1)} 
\ar@{..>}[dl]_{-i_{BP}} \ar[r]^{j_H}\ar[d]^\delta &
H^{[t+1]}\wedge\oBP^{\wedge(s+1)} \ar[d]^{\delta}\\
H^{[t]}\wedge\oBP^{\wedge s} \ar[r]^{k_H} & 
\oH^{\wedge(t+1)}\wedge\oBP^{\wedge s} & 
\oH^{\wedge(t+2)}\wedge\oBP^{\wedge s} \ar[l]_{i_H} \ar[r]^{j_H} &
H^{[t+2]}\wedge\oBP^{\wedge s} 
}
\label{adamsd2}
\ee
We will construct an element 
$y_0\in\pi_{u+t+1}(\oH^{\wedge(t+1)}\wedge\oBP^{\wedge(s+1)})$ 
such that
$i_H\delta y_0=k_Hz'$. We can then take $y'=\delta y_0$,
so the representative $j_Hy'$ for $d_2^Hx$ lifts across $\delta$, to the 
cocyle $j_Hy_0$. This proves the first claim of the theorem.

To construct $y_0$ we first show that $k_Hz'$ lifts across 
$i_{BP}$, and then that
the dotted triangle commutes: $i_H\delta=-i_{BP}$.

The first step is organized by the diagram
\[
\xymatrix{
& H^{[t]}\wedge \oBP^{\wedge s} \ar[r]^{j_{BP}} \ar[d]^{k_H} &
H^{[t]}\wedge BP^{[s]} \ar[d]^{k_H} \ar@/^5em/[dd]^{d^H_1} \\
\oH^{\wedge(t+1)}\wedge\oBP^{\wedge(s+1)} \ar[r]^{i_{BP}} & 
\oH^{\wedge(t+1)}\wedge\oBP^{\wedge s} \ar[r]^{j_{BP}} \ar[d]^{j_H} &
\oH^{\wedge(t+1)}\wedge BP^{[s]} \ar[d]^{j_H} \\
& H^{[t+1]}\wedge\oBP^{\wedge s} \ar[r]^{j_{BP}} &
H^{[t+1]}\wedge BP^{[s]}
}
\]
We have $d^H_1j_{BP}z'=j_{BP}d^H_1z'=0$.
Since the Adams spectral sequence for $BP^{[s]}$ collapses at $E_2$,
this implies that $k_Hj_{BP}z'=0$, so $j_{BP}k_Hz'=0$ and hence
that $k_Hz'$ lifts to some 
$y_0\in\pi_{u+t+1}(\oH^{\wedge(t+1)}\wedge\oBP^{\wedge(s+1)})$. 
It is important to note that the only property of $y_0$ used in this
part of the proof is that $i_{BP}y_0=k_Hz'$. 

Since $i_H\delta=i_{BP}:\oBP\to S^1$, it may appear that the triangle 
commutes without the sign. But $\delta$ has many meanings. 
The triangle expands to the perimeter of the diagram
\[
\xymatrix{
\oH\wedge Y\wedge X\wedge\oBP \ar[rr]^{1\wedge i_{BP}} \ar[d]^{(34)} &&
\oH\wedge Y\wedge X\wedge S^{-1} \ar[d]^{(34)} \ar@/^5em/[dd]^{(134)} \\
\oH\wedge Y\wedge\oBP\wedge X \ar[d]^{1\wedge\delta\wedge1} 
\ar[rr]^{1\wedge i_{BP}\wedge1} \ar@{}[drr]|{-} &&
\oH\wedge Y\wedge S^{-1}\wedge X \ar[d]^{(13)} \\
\oH\wedge Y\wedge\oH\wedge X \ar[rr]^{i_H\wedge1} && 
S^{-1}\wedge Y\wedge\oH\wedge X
}
\]
in which $X=\oBP^{\wedge s}$ and $Y=\oH^{\wedge t}$, and the cycles indicate
the appropriate permutations of factors. The top inner square commutes, and
it remains to check that the bottom square commutes up to the indicated sign.

For this, note that we can drop the terminal $X$. We can also move the $Y$ 
to the right end and drop it; so we need to check that 
\be
\xymatrix{
\oH\wedge\oBP \ar[rr]^{1\wedge i_{BP}} \ar[d]^{1\wedge\delta} \ar@{}[drr]|{-}
&& \oH\wedge S^{1} \ar[d]^{(12)} \\
\oH\wedge\oH \ar[rr]^{i_H\wedge1} && S^{1}\wedge\oH 
}
\label{sum-maps}
\ee
commutes up to sign. To verify this, first map out to $S^{1}\wedge S^{1}$:
\[
\xymatrix{
\oH\wedge\oBP \ar[rr]^{1\wedge i_{BP}} \ar[d]^{1\wedge\delta} \ar@{}[drr]|{-} 
&& \oH\wedge S^{1} \ar[d]^{(12)} \ar[dr]^{i_H\wedge1} \\
\oH\wedge\oH \ar[rr]^{i_H\wedge1} \ar[drrr]_{i_H\wedge i_H} && 
S^{1}\wedge\oH \ar[dr]^{1\wedge i_H} & S^{1}\wedge S^{1} \ar[d]^{(12)} \\
&&& S^{1}\wedge S^{1} 
}
\]
Now $i_H\delta=i_{BP}$, and
the switch map on $S^{1}\wedge S^{1}$ multiplies by $-1$, so the outer
diagram commutes up to sign. Thus the sum of the two maps across 
\eqref{sum-maps} lifts through a map $\oH\wedge\oBP\to S^{1}\wedge H$. 
But $H^1(\oH\wedge\oBP)=0$.

This completes the proof of the first part of the theorem. To relate the
Adams differential to the algebraic Novikov differential, recall from
\eqref{CEfilt} and \eqref{zprime} 
that in the algebraic Novikov $E_1$ term 
\begin{quote}
 $x\in F_{CE}^sE_2^{s+t,u+t}(\STop;H)$ is represented 
the class of $j_{BP}z'\in\pi_{u+t}(\oH^{[t]}\wedge\oBP^{[s]})$.
\end{quote}
Similarly, from \eqref{adamsd2},
\begin{quote}$d^A_2x\in F_{CE}^{s+1}E_2^{s+t+2,u+t+1}(\STop;H)$ is 
represented by $j_{BP}j_Hy_0\in\pi_{u+t+1}(H^{[t+1]}\wedge BP^{[s+1]})$.
\end{quote}
We need to see that these two classes are also related by $d_1^{AN}$. 

The computation of the differential $d^{AN}_1$ is organized by the top of the
following diagram. 
\[
\xymatrix@C-3em{
H^{[t]}\wedge BP^{[s]} \ar[rr]^{d^{BP}_1} &&
H^{[t]}\wedge BP^{[s+1]} &&
H^{[t+1]}\wedge BP^{[s+1]} \ar[ll]_{i_H} \\
\oH^{\wedge t}\wedge BP^{[s]} 
\ar[u]^{j_H} \ar[rr]^{d^{BP}_1} \ar[dr]_{k_{BP}} &&
\oH^{\wedge t}\wedge BP^{[s+1]} \ar[u]^{j_H} &&
\oH^{\wedge(t+1)}\wedge BP^{[s+1]} \ar[u]^{j_H} \ar[ll]_{i_H} \\
& \oH^{\wedge t}\wedge\oBP^{\wedge(s+1)} \ar[ur]_{j_{BP}} &&
\oH^{\wedge(t+1)}\wedge\oBP^{\wedge(s+1)} \ar[ll]_{i_H} \ar[ur]_{j_{BP}}
}
\]
Let $a\in H^{t,u}(P;Q^s)$ be an element in the algebraic Novikov $E_1$ term.
To compute $d^{AN}_1a$, find a representative 
$a'\in\pi_{u+t}(H^{[t]}\wedge BP^{[s]})$. By the collapse at $E_2$ of the 
Adams spectral sequence for $BP^{[s]}$, $a'$ lifts to an element 
$y_1\in\pi_{t+u}(\oH^{\wedge t}\wedge BP^{[s]})$. Since $a'$ is a cocycle,
$d^{BP}_1y_1=i_Hb\in\pi_{u+t}(\oH^{\wedge t}\wedge BP^{[s+1]})$ for
some $b\in\pi_{u+t+1}(\oH^{\wedge(t+1)}\wedge BP^{[s+1]})$. 
Then $j_Hb$ is a representative for $d^{AN}_1a$.

With $a'=j_{BP}z'$, our earlier work suggests that we might choose
$b=j_{BP}y_0$; then $j_Hb=j_Hj_{BP}y_0=j_{BP}j_Hy_0$ would be our 
representative for $d^{AN}_1a$, reaching the conclusion we want.
For this to work, we need $j_{BP}y_0$ to satisfy the equation required of $b$;
that is, $i_Hj_{BP}y_0=d^{BP}_1y_1=j_{BP}k_{BP}y_1$.
Since $i_Hj_{BP}=j_{BP}i_H$, this will be guaranteed if 
\[
i_Hy_0=k_{BP}y_1\in\pi_{u+t}(\oH^{\wedge t}\wedge\oBP^{\wedge(s+1)})\,.
\]
But this is precisely what is guaranteed to us by the application
of May's lemma \ref{lemma-may} to the following diagram. 
\[
\xymatrix{
\oH^{\wedge t}\wedge\oBP^{\wedge(s+1)} \ar[r] \ar[d] &
\oH^{\wedge t}\wedge\oBP^{\wedge s} \ar[r] \ar[d] &
\oH^{\wedge t}\wedge BP^{[s]} \ar[d]^{j_H} \ar@/_3em/[ll]_{k_{BP}} \\
H^{[t]}\wedge\oBP^{\wedge(s+1)} \ar[d] \ar[r] &
H^{[t]}\wedge\oBP^{\wedge s} \ar[r]^{j_{BP}} \ar[d]^{k_H} &
H^{[t]}\wedge BP^{[s]} \ar[d]^{k_H} \\
\oH^{\wedge(t+1)}\wedge\oBP^{\wedge(s+1)} 
\ar[r]^{i_{BP}} \ar@/^5em/[uu]^{i_H} &
\oH^{\wedge(t+1)}\wedge\oBP^{\wedge s} \ar[r]^{j_{BP}} &
\oH^{\wedge(t+1)}\wedge BP^{[s]} 
}
\]
This concludes the proof. 
\end{proof}

\subsection{Motivic conclusions.}

This proof works in much greater generality. As in \cite{miller1981relations},
the square of spectral sequences can be set up for any map of ring spectra
$A\to B$ and any spectrum $X$ for which the $B$-Adams spectral 
sequence 
\[
E_2(A\wedge\overline{A}{}^{\wedge s}\wedge X;B)\implies
\pi_*(A\wedge\overline{A}{}^{\wedge s}\wedge X)
\]
converges and collapses at the $E_2$-page for all $s$. The proof holds whenever $B^1(\overline{B}\wedge\overline{A})=0$.

In particular, the proof works in the motivic context, at least over $\CC$, 
with $H$ replaced by the mod 2 motivic Eilenberg Mac Lane spectrum (also 
denoted $H$)
and $BP$ by the spectrum $BPM$ considered by Hu, Kriz, and Ormsby 
\cite{hu2011remarks}. They observe that as a comodule for $A_\Mot$
\[
H_*(BPM)=\A_\Mot\square_E\V\,,
\]
where $A_\Mot$ is the motivic dual Steenrod algebra
\[
A_{\Mot}=V[\tau_0,\tau_1,\ldots,\xi_1,\xi_2,\ldots]/(\tau_n^2=\tau\xi_{n+1})\,,
\]
\[
|\tau_n|=(2^{n+1}-1,2^n-1),\ |\xi_n|=(2^{n+1}-2,2^n-1)\,,
\]
and $E$ is the exterior algebra over $\V$ generated by $\tau_0,\tau_1\ldots$
By change of rings, then, 
\[
E_2(BPM;H)=\F_2[\tau,v_0,v_1,\ldots]
\]
where $v_i$ is the class of the primitive $\tau_i$.  
The motivic Adams spectral sequence for $BPM$ thus collapses. 

We also need to check that $H^{1,0}(\overline{BPM}\wedge\oH)=0$. This follows
from the fact that $BPM$ and $H$ are cellular (\!\!\cite{hu2011remarks}), 
using the computations of $H_*(BPM)$ and $H_*(H)$ and the K\"unneth theorem
\cite{dugger2007motivic}.

We have
\[
h_0\in E_2^{1,2,1}(\SMot;H)\,,\quad 
v_n\in E_2^{1,2^{n+1}-1,2^n-1}(\SMot;H)\,,
\]
and the differentials
constructed in Proposition \ref{localg.nss} produce the following differentials in the motivic Adams spectral sequence:
\[d_2v_{n+1}\equiv v_n^2h_0,\ n\geq 2\]
modulo terms of higher Cartan-Eilenberg filtration. 
Recalling that $\deg=t-2w=1$ we see $\deg(h_0)=0$ and $\deg(v_n)=1$. Thus $v_n^2h_0$ has Chow degree $2$ and Adams filtration $3$ and any other such element must be of the form $v_iv_jh_0$. We find that 
\[v_n^2h_0\in E_2^{3,2^{n+2},2^{n+1}-1}(\SMot;H)\]
is the only element in its trigrading so, in fact, our calculation has no indeterminancy. This is our last theorem.

\begin{thm} 
In the $\eta$-localized motivic Adams spectral sequence,
\[
d_2v_{n+1}=v_n^2h_0\,,\quad n\geq 2\,.
\]
\end{thm}

\bibliographystyle{amsplain}

\bibliography{andrews-miller.bib}{}

\Addresses

\end{document}